\documentclass{article}

\pdfoutput=1
\usepackage{amsmath}
\usepackage{wrapfig}
\usepackage{subcaption}
\usepackage{color}
\usepackage[dvipsnames]{xcolor}
\usepackage{amsthm}
\usepackage{amssymb}
\usepackage{bm}
\usepackage{multirow}
\usepackage{siunitx}
\usepackage{graphicx}
\usepackage{doi}
\usepackage{hyperref}
\usepackage[normalem]{ulem}

\usepackage[numbers]{natbib}

\title{Robust and conservative dynamical low-rank methods for the Vlasov equation via a novel macro-micro decomposition}
\author{Jack Coughlin\footnote{Department of Applied Mathematics, University of Washington, Seattle, WA 98195, USA (johnbc@uw.edu).} \ \ , \ Jingwei Hu\footnote{Department of Applied Mathematics, University of Washington, Seattle, WA 98195, USA (hujw@uw.edu).} \ \ , \ Uri Shumlak\footnote{Aerospace and Energetics Research Program, University of Washington, Seattle, WA 98195, USA (shumlak@uw.edu).}}

\numberwithin{equation}{section}

\newtheorem{thm}{Theorem}[section]
\newtheorem{prop}{Proposition}[section]
\newtheorem{corollary}{Corollary}[section]
\theoremstyle{definition}
\newtheorem{remark}{Remark}

\newcommand{\Dx}{{\partial_{x}}}
\newcommand{\Dv}{{\partial_{v}}}
\newcommand{\qm}{}

\newcommand{\Lwinv}{{w^{-1}(v)}}
\newcommand{\xLwinv}{{x, w^{-1}(v)}}
\newcommand{\ppp}{{P_\Phi^\perp}}

\newcommand{\Ampere}{Amp\`{e}re}

\begin{document}

\maketitle

\begin{abstract}
Dynamical low-rank (DLR) approximation has gained interest in recent years as a viable solution to
the curse of dimensionality in the numerical solution of kinetic equations including the Boltzmann
and Vlasov equations. These methods include the projector-splitting and Basis-update \& Galerkin (BUG) DLR
integrators, and have shown promise at greatly improving the computational efficiency of kinetic
solutions. However, this often comes at the cost of conservation of charge, current and energy. 
In this work we show how a novel macro-micro decomposition may be used to separate the distribution
function into two components, one of which carries the conserved quantities, and the other of which
is orthogonal to them.
We apply DLR approximation to the latter, and thereby achieve a clean and extensible approach to a
conservative DLR scheme which retains the computational advantages of the base scheme.
Moreover, our approach requires no change to the mechanics of the DLR approximation, so it
is compatible with both the BUG family of integrators and the projector-splitting integrator which we use here.
We describe a first-order integrator which can exactly conserve charge and either current or energy,
as well as an integrator which exactly conserves charge and energy and exhibits second-order accuracy on our test problems.
To highlight the flexibility of the proposed macro-micro decomposition, we implement a pair of velocity
space discretizations, and verify the claimed accuracy and conservation properties on a suite of plasma
benchmark problems.
\end{abstract}

{\small 

}

\section{Introduction}

Kinetic equations describe the behavior of rarefied gases and plasmas at low to moderate levels of collisionality.
Plasma physics in particular is dominated by the study of low collisionality regimes, where the full six-dimensional kinetic physics plays a role in the development of micro-instabilities and associated turbulence, anomalous transport, and wave propagation.
Unfortunately, the numerical solution of kinetic equations is extremely costly due to their high dimensionality:
memory and computational requirements both scale as $\mathcal{O}(N^6)$ for a discretization with $N$ degrees of freedom in each dimension.
This obstacle is commonly known as the curse of dimensionality.

Particle-in-cell methods \cite{birdsallPlasmaPhysicsComputer1985} sidestep the curse of dimensionality by taking a Lagrangian, particle-based approach to modeling phase space,
although these suffer from statistical noise.
More recently, dimension-reduction techniques based on low-rank approximation to the solution have seen success in solving kinetic equations.
Such approaches propose to model the particle distribution function $f$ as a low-rank combination of lower-dimensional quantities,
thereby greatly reducing computational and memory requirements.
These low-rank methods include the ``step-truncation'' approach \cite{guoLowRankTensor2022} and dynamical low-rank (DLR) approximation \cite{kochDynamicalLowRank2007}.
Dynamical low-rank approximation has been successfully applied to kinetic models for 
neutral gas dynamics \cite{huAdaptiveDynamicalLow2022, einkemmerEfficientDynamicalLowRank2021},
radiative transport \cite{einkemmerAsymptoticpreservingDynamicalLowrank2021, pengLowrankMethodTwodimensional2020},
and plasma physics problems \cite{einkemmerLowRankProjectorSplittingIntegrator2018, einkemmerAcceleratingSimulationKinetic2023, coughlinEfficientDynamicalLowrank2022}.

An obstacle to usefully applying DLR approximation to kinetic equations is the preservation of the equation's conservation properties.
The truncation implied by the low-rank approximation does not necessarily respect the conservation of the physical observables of
mass (often called charge density in the plasma setting), momentum (current), and total energy.
This conservation failure is inherent to the low-rank approximation and occurs independently of the physical or time discretization chosen.

In \cite{einkemmerQuasiConservativeDynamicalLowRank2019}, it was shown how to use a method based on Lagrange multipliers to obtain a
quasi-conservative scheme for the Vlasov equation of plasma dynamics in the DLR framework.
A pair of papers \cite{einkemmerMassMomentumEnergy2021, einkemmerRobustConservativeDynamical2023} demonstrated a first-principles
way of achieving conservation in the DLR framework, by forcing the velocity basis to contain the functions
$(1, v, |v|^2/2)$ via a modification to the DLR Galerkin condition.
The second of these works \cite{einkemmerRobustConservativeDynamical2023} showed how to achieve this in the context of the 
Basis-update \& Galerkin (BUG) integrator \cite{cerutiUnconventionalRobustIntegrator2022}, which is robust to the presence of small singular values.
Lately it has been observed that the rank-adaptive BUG integrator is also conservative if equipped
with a conservative rank-truncation algorithm \cite{einkemmerConservationPropertiesAugmented2023}.

In \cite{pengHighorderLoworderHOLO2021} the DLR method was used to evolve the high-order component of a high-order/low-order (HOLO) scheme for the radiative transport equation.
The DLR method was employed as a moment closure method for a low-order fluid system, while a least-squares projection was applied to keep its conserved
moments close to those of the fluid system: in this way a conservative DLR method for the radiative transport equation was achieved.
Finally, under the step truncation family of methods, conservative projections have been used to obtain overall conservative low-rank solutions to the 
Vlasov equation \cite{guoLocalMacroscopicConservative2023, guoConservativeLowRank2022, guoLocalMacroscopicConservative2022}.

Our method is closest in spirit to \cite{koellermeierMacromicroDecompositionConsistent2023}, which shows how to use a modal Legendre discretization
of a kinetic extension of the shallow-water equations to combine conservation with the BUG integrator.
In that work, the authors apply a low-rank ansatz to the trailing ``microscopic'' modes, while evolving the leading modes using standard conservative techniques.
However, as formulated, the method is limited to modal discretizations of the phase space coordinate.

In this paper we show a new way of obtaining a conservative DLR method for kinetic equations, with a focus on the Vlasov-Fokker-Planck equation.
Our method is based on a novel macro-micro decomposition that operates at the equation level and is designed to be amenable to low-rank approximation.
The macro part of the decomposition may be solved using standard conservative discretization techniques, while we apply a DLR approximation
to the micro part.
The benefits of our formulation include that it is compatible with various DLR integrators, including the projector-splitting integrator
which can be formally extended to second-order accuracy via standard Strang splitting method,
and does not require any rank augmentation at intermediate steps.
Moreover, our macro-micro decomposition and subsequent DLR approximation are independent of the velocity space discretization, which may be chosen last.
This represents a substantial benefit for plasma applications, where shock-capturing discretizations in velocity such as Discontinuous Galerkin
are very popular for their ability to resolve fine phase space structures \cite{hoPhysicsBasedAdaptivePlasmaModel2018, hakimAliasFreeMatrixFreeQuadratureFree2020}.

The rest of this paper is organized as follows.
In the remainder of this section we introduce the model equation of plasma physics we will discretize, namely the Vlasov equation with Dougherty-Fokker-Planck collision operator.
In Section \ref{sec:macro-micro} we describe our novel macro-micro decomposition and derive the equations of evolution for each part.
In Section \ref{sec:time-discretization} we describe the time discretization of our macro-micro equations in the DLR framework and
prove the claimed conservation properties.
Sections \ref{sec:spatial-discretization} and \ref{sec:velocity-discretization} describe our chosen discretizations of physical and velocity space, respectively.
Finally, we present numerical results on standard benchmark problems for the Vlasov equation in Section \ref{sec:numerical_results}. 
The paper is concluded in Section \ref{sec:conclusion}.

\subsection{Properties of the Vlasov equation with Dougherty collisions}
In this work, we consider a kinetic equation for a single-species plasma with collisions, 
known as the electrostatic Vlasov equation with the Dougherty-Fokker-Planck or Lenard-Bernstein collision operator \cite{doughertyModelFokkerPlanckEquation1964}.
In physical terms, this equation describes the motion of an electron species against a static ion background, neglecting the influence of magnetic fields.
In dimensionless form, the electrostatic Vlasov equation for a single species in $d$ physical and velocity space dimensions is
\begin{align}
    \label{eqn:vlasov}
\partial_t f + v \cdot \nabla_x f + E \cdot \nabla_v f = Q(f), \quad t > 0, \quad x \in \Omega_x \subset \mathbb{R}^d, \quad v \in \mathbb{R}^d.
\end{align}
Note that our non-dimensionalization also reverses the charge convention, so that the dynamic electron species is given a positive unit charge.
The function $f(x, v, t)$ is the normalized probability density function, representing the density of particles with
velocity $v$, at position $x$ and time $t$.
The evolution of the electric field $E(x, t)$ is coupled to the charge density $\rho$ by either Gauss's law,
\begin{align}
    \label{eqn:gauss}
    \nabla_x \cdot E(x, t) = \rho(x, t) - \rho_0, \quad \rho(x, t) = \int_{\mathbb{R}^d} f\,\mathrm{d} v,
\end{align}
with $\rho_0$ a uniform background density satisfying $\int_{\Omega_x} \rho(x,t) - \rho_0\,\mathrm{d} x = 0$,
or by \Ampere's law:
\begin{align}
    \label{eqn:ampere}
    \partial_t E(x, t) = -J(x, t), \quad J(x, t) = \int_{\mathbb{R}^d} v f\,\mathrm{d} v.
\end{align}
The initial electric field is chosen to satisfy Gauss's law, and it is easy to show that at the continuous level,
if Gauss's law is satisfied at \( t=0 \) and \( E \) evolves according to \eqref{eqn:ampere}, then Gauss's law is satisfied for all time.
However, if one is not careful, numerical discretization errors can lead to violations of \eqref{eqn:gauss} at the discrete level.
These so-called divergence errors can lead to unphysical solutions to certain problems.
However, Ampere's law has the advantage that it keeps the overall hyperbolic nature of the system of equations, and ``divergence cleaning'' 
strategies have been developed \cite{munzThreedimensionalFinitevolumeSolver2000} to ameliorate the issue of numerical divergence errors.

The Dougherty-Fokker-Planck collision operator for one species, $Q(f)$, is defined as
\begin{align*}
    Q(f) = \nu \nabla_v \cdot (T \nabla_v f + (v - u) f).
\end{align*}
Here $\nu$ is a dimensionless collision frequency and $T(x, t)$ and $u(x, t)$ are the local temperature and fluid velocity, defined via moments of $f$:
\begin{align}
    \quad u(x, t) = \frac{1}{\rho} \int_{\mathbb{R}^d} vf\,\mathrm{d} v, 
    \quad T(x, t) = \frac{1}{d\rho} \int_{\mathbb{R}^d} |v - u|^2 f \, \mathrm{d} v.
\end{align}
It is not hard to show \cite{doughertyModelFokkerPlanckEquation1964} that $Q(f)$ satisfies the identity
\begin{align*}
    \int_{\mathbb{R}^d} \bm{\phi}(v) Q(f)\,\mathrm{d} v = \mathbf{0},
\end{align*}
where $\bm{\phi}(v) = (1, v, |v|^2/2)^T$ is the vector of so-called collision invariants.
These correspond respectively to conservation of mass, momentum, and energy in elastic interparticle collisions.
Each collision invariant therefore admits a local conservation law, which we now derive.
Define the current density $J$, kinetic energy density $\kappa$, and the (total) energy density $e$ as
\begin{align}
    J(x, t) &= \int_{\mathbb{R}^d} vf\,\mathrm{d} v = \rho u, \\
    \kappa(x, t) &= \int_{\mathbb{R}^d} \frac{|v|^2}{2} f\,\mathrm{d} v,\\
    e(x, t) &= \kappa + \frac{|E|^2}{2}.
\end{align}
We note that what we have called the charge and current densities, $\rho$ and $J$, are in fact identical to
the mass density and momentum for the dimensionless, single-species Vlasov equation we consider here.
In the case of multiple species it is total momentum, not current, which is conserved.
Here however, to avoid naming the same quantity two different ways, we refer to conservation of current.\footnote{The alternative is to refer to the source term in \Ampere's law as the momentum.}
Now, taking the moments of \eqref{eqn:vlasov} weighted by each component of $\bm{\phi}(v)$ gives
a system of local conservation laws for $(\rho, J, e)^T$:
\begin{align}
    \label{eqn:local_conservation_law-1}
    \partial_t \rho + \nabla_x \cdot J &= 0, \\
    \label{eqn:local_conservation_law-2}
    \partial_t J + \nabla_x \cdot \mathbf{\sigma} &= \rho E, \\
    \label{eqn:local_conservation_law-3}
    \partial_t e + \nabla_x \cdot \mathbf{q} &= 0,
\end{align}
where $\mathbf{\sigma} = \int_{\mathbb{R}^d} (v \otimes v) f \,\mathrm{d} v$ and $\mathbf{q} = \frac{1}{2} \int_{\mathbb{R}^d} v |v|^2\,\mathrm{d} v$.
An important goal of numerical discretizations of the Vlasov equation is the preservation of 
the local conservation laws \eqref{eqn:local_conservation_law-1}-\eqref{eqn:local_conservation_law-3}.
Failure to respect charge or energy conservation can lead to numerical instabilities and nonphysical solutions.
More fundamentally, it is precisely these ``observables'' which are often of greatest interest to the practitioner,
since the purpose of numerical solutions to kinetic equations is often to shed light on the transport and partition
of density and energy.

If we take the spatial domain $\Omega_x$ to be periodic and integrate in $x$, then the divergence terms 
in \eqref{eqn:local_conservation_law-1}-\eqref{eqn:local_conservation_law-3} vanish, and we obtain a set of
global conservation laws for charge, current and energy:
\begin{align}
\label{eqn:global_conservation_law-1}
    \partial_t \int_{\Omega_x} \rho \,\mathrm{d}x &= 0, \\
    \partial_t \int_{\Omega_x} e \, \mathrm{d}x &= 0,
\end{align}
and
\begin{align}
\label{eqn:global_conservation_law-2}
    \partial_t \int_{\Omega_x} J \,\mathrm{d}x &= \int_{\Omega_x} (\nabla_x \cdot E + \rho_0) E \, \mathrm{d}x = \int_{\Omega_x} \frac{1}{2} \nabla_x |E|^2 + \rho_0 (\nabla_x \varphi)\,\mathrm{d} x = 0,
\end{align}
where we have used the fact that $E = -\nabla_x \varphi$ for some potential function $\varphi$.


\section{A novel macro-micro-decomposition of the Vlasov equation}
\label{sec:macro-micro}
In this section we describe the novel macro-micro decomposition at the core of our method.
By macro-micro decompositions, we refer to a family of methods which use a decomposition of the form
\begin{align*}
f(x, v, t) = \mathcal{N}(x, v, t) + g(x, v, t),
\end{align*}
in which $\mathcal{N}$ is designed to share its first $d+2$ moments with $f$, and those same moments of $g$ vanish identically.
The strategy is then to evolve $\mathcal{N}$ as accurately as possible using standard conservative discretizations.
On the other hand, the fact that $g$ does not contribute to the conserved quantities lets us evolve it at the precision
dictated by the kinetic physics of the problem.
Thus, the macro-micro decomposition splits a high-dimensional problem into two parts: a lower dimensional problem
for $\mathcal{N}$ which can be solved conservatively, and a high-dimensional problem for $g$ to which we can apply either
a coarser discretization or more sophisticated dimension reduction techniques.

Perhaps the most natural and widely-known macro-micro decomposition for a collisional kinetic equation is given in \cite{bennouneUniformlyStableNumerical2008}.
This approach takes the equilibrium distribution function, the Maxwellian $\mathcal{M}$, as the ``macro'' component:
\begin{align*}
    \mathcal{M}(x, v, t) = \frac{\rho(x, t)}{\left( 2\pi T(x, t)\right)^{d/2}} e^{-\frac{|v - u(x, t)|^2}{2T(x, t)}}.
\end{align*}
For strongly collisional problems, such a decomposition can expect the remainder $g$ to be small compared to the Maxwellian.
However, this decomposition is less appealing when collisions are weaker.
Moreover, it is not favored by the DLR method that we use here.
A key step in DLR is the projection of the right-hand side onto the tangent space of the low-rank approximate solution manifold \cite{kochDynamicalLowRank2007}.
The problem is that for a Maxwellian-based macro-micro decomposition, computing the necessary projection requires evaluating
integrals such as
\begin{align}
    \label{eqn:maxwellian-weighted-integral}
    \left( \frac{\rho(x, t)}{2\pi T(x, t)} \right)^{d/2} \int_{\mathbb{R}^d} V_j(v) e^{-\frac{|v - u(x, t)|^2}{2T(x, t)}} \, \mathrm{d} v,
\end{align}
which require $O(N^{2d})$ operations in the general case, where $N$ is the number of degrees of freedom in each spatial and velocity dimension.
References \cite{einkemmerEfficientDynamicalLowRank2021, coughlinEfficientDynamicalLowrank2022} present DLR methods for
highly collisional regimes based on an expansion around the Maxwellian.
Both of these papers deal with isothermal flow, in which case the Maxwellian-weighted integrals can be efficiently evaluated
by exploiting convolutional structure.

Our proposed method is aimed at plasma applications at the electron scale, where collisions are typically much weaker than in neutral gases,
and phase space often exhibits highly non-equilibrium features.
As such, we do not expect a Maxwellian-based macro-micro decomposition such as \cite{bennouneUniformlyStableNumerical2008},
nor a Maxwellian-centered DLR scheme such as \cite{einkemmerEfficientDynamicalLowRank2021, coughlinEfficientDynamicalLowrank2022},
to be advantageous.
Rather, to avoid the difficulties presented by integrals such as \eqref{eqn:maxwellian-weighted-integral}, we derive
a macro component $\mathcal{N}$ with fixed rank of $d+2$, which allows for efficient computation of the projections
required by DLR approximation.

To illustrate the ideas, we restrict our discussion to the ``1D1V'' case of $d=1$.
At this point we also modify the initial-boundary value problem for the Vlasov equation by possibly truncating velocity space.
We let $v \in \Omega_v$, where $\Omega_v \subseteq \mathbb{R}$ may be either bounded or equal to $\mathbb{R}$.
Our one-dimensional, truncated Vlasov equation is therefore
\begin{align}
    \label{eqn:vlasov-1d1v}
\partial_t f + v \partial_x f + E \partial_v f = Q(f), \quad t > 0, \quad x \in \Omega_x \subset \mathbb{R}, \quad v \in \Omega_v \subseteq \mathbb{R}.
\end{align}
Denote the $(x, v)$ domain by $\Omega = \Omega_x \times \Omega_v$.
We impose periodic boundary conditions in $x$.
For an unbounded velocity domain, no boundary conditions in $v$ are required, although we must assume
that $f$ decays sufficiently quickly as $v \rightarrow \pm\infty$.
On a bounded velocity domain we make the same assumption of rapid decay, so that $f$ and its
derivatives are negligible at the velocity boundary.
The fluid variables must be redefined in terms of moments of $f$ over $\Omega_v$, so we will write
\begin{align}
    \label{eqn:rhoJe-def}
\rho = \langle f \rangle_v, \quad J = \left\langle vf \right\rangle_v, \quad \kappa = \left\langle \frac{|v|^2}{2} f \right\rangle_v, \quad e = \kappa + \frac{|E|^2}{2},
\end{align}
\begin{align}
    \label{eqn:uT-def}
\quad u = \frac{J}{\rho}, \quad T = \frac{1}{\rho} \left\langle |v - u|^2 f \right\rangle_v,
\end{align}
where $\left\langle \cdot \right\rangle_v = \int_{\Omega_v} \cdot \, \mathrm{d} v$.
With these definitions we can again show that 
\begin{align}
    \label{eqn:Q-conserves}
\left\langle \bm{\phi}(v) Q(f) \right\rangle_v = \mathbf{0},
\end{align}
using the fast decay of $f$ and its derivatives at the velocity boundary.
Therefore, the derivation of the local conservation laws \eqref{eqn:local_conservation_law-1}-\eqref{eqn:local_conservation_law-3} holds, as well as the global conservation laws \eqref{eqn:global_conservation_law-1}-\eqref{eqn:global_conservation_law-2}.

Our macro-micro decomposition is based on orthogonal projection in an inner product space over $\Omega_v$.
To work in an inner product space over the possibly unbounded domain \( \Omega_v \), we require a weight function which we denote \( w(v): \Omega_v \mapsto \mathbb{R} \).
This induces a pair of weighted inner products on \( \Omega_v \) and \( \Omega \), respectively:
\begin{align*}
\left\langle g, h \right\rangle_\Lwinv = \int_{\Omega_v} w^{-1}(v) g(v) h(v)\,\mathrm{d} v, \qquad
\left\langle g, h \right\rangle_\xLwinv = \int_{\Omega} w^{-1}(v) g(x, v) h(x, v)\, \mathrm{d} x \mathrm{d} v.
\end{align*}
We denote the corresponding inner product spaces by \( L^2(\Omega_v, w^{-1}) \) and \( L^2(\Omega, w^{-1}) \) respectively.
From standard theory \cite{gautschiOrthogonalPolynomialsComputation2004}, the inner product \( \left\langle \cdot, \cdot \right\rangle_\Lwinv \)
has an associated family of orthonormal polynomials, which we will denote by \( p_n(v) \).
Examples of classical orthonormal polynomial families include the Hermite polynomials
with weight function \( \frac{1}{\sqrt{2\pi} v_0} e^{-(v/v_0)^2/2} \) on the whole real line, and the scaled Legendre polynomials,
which have the constant weight function \( w(v) = \frac{1}{v_{max}} \) on the domain \( [-v_{max}, v_{max}] \).

All families of orthonormal polynomials satisfy an orthogonality relation
\begin{align}
    \label{eqn:p-ortho}
\int_{\Omega_v} w(v) p_n(v) p_m(v)\,\mathrm{d} v = \left\langle w(v) p_n(v), w(v) p_m(v) \right\rangle_\Lwinv = \delta_{nm},
\end{align}
and a symmetric three-term recurrence relation
\begin{align}
    \label{eqn:p-recurrence}
v p_n(v) = a_n p_{n+1}(v) + b_n p_n(v) + a_{n-1} p_{n-1}(v).
\end{align}
Also define the coefficients \( d_{10}, d_{20}, d_{21} \) by 
\begin{align}
\label{eqn:p-derivatives}
p'_1(v) = d_{10} p_0(v), \quad \text{and} \quad p'_2(v) = d_{20} p_0(v) + d_{21} p_1(v).
\end{align}

The first three orthonormal polynomials \( \mathbf{p}(v) = (p_0(v), p_1(v), p_2(v))^T \) are related to the collision
invariants \( \bm{\phi}(v) \) by a lower-triangular, invertible matrix \( C \):
\begin{align}
    \label{eqn:phi_p_relation_c}
    \bm{\phi}(v) = \begin{pmatrix}
    1 \\ v \\ v^2/2
    \end{pmatrix}
    =
    \begin{pmatrix}
    c_{00} \\
    c_{10} & c_{11} \\
    c_{20} & c_{21} & c_{22}
    \end{pmatrix}
    \begin{pmatrix}
    p_0(v) \\ p_1(v) \\ p_2(v)
    \end{pmatrix}
    = 
    C \begin{pmatrix}
    p_0(v) \\ p_1(v) \\ p_2(v)
    \end{pmatrix}
    = C \mathbf{p}(v).
\end{align}
The span of \( w(v) \bm{\phi}(v) \) is an important subspace, which we will denote by \( \Phi \):
\begin{align*}
\Phi = \text{span} \{ w(v), vw(v), |v|^2 w(v)/2 \}.
\end{align*}
We are interested in the orthogonal projection onto \( \Phi \) with respect to \( \left\langle \cdot \right\rangle_\Lwinv \), which we will denote \( P_\Phi \).
The existence of the invertible matrix \( C \) shows that \( w(v) \mathbf{p}(v) \) is a basis for \( \Phi \).
In fact, it is an orthogonal basis, since its elements are orthonormal per \eqref{eqn:p-ortho}.
This gives an explicit formula for the orthogonal projection \( P_\Phi \):
\begin{align}
    \label{eqn:P_phi_def}
P_\Phi f = w(v) \mathbf{p}(v)^T \left\langle w(v) \mathbf{p}(v), f \right\rangle_\Lwinv = w(v) \mathbf{p}(v)^T \left\langle \mathbf{p}(v) f \right\rangle_v.
\end{align}
Denote the orthogonal complement of \( P_\Phi \) by \( P_\Phi^\perp = I - P_\Phi \).
Our proposed macro-micro decomposition of \( f \) is that induced by the pair of projections: 
\begin{align*}
f(x, v, t) = \underbrace{P_\Phi f(x, v, t)}_{\mathcal{N}} + \underbrace{P_\Phi^\perp f(x, v, t)}_{g}.
\end{align*}
The function \( \mathcal{N} \) has an explicit formula in terms of \( p_0, p_1, p_2 \) and the corresponding moments:
\begin{align}
    \label{eqn:defn-N}
    \mathcal{N}(x, v, t) = w(v) \mathbf{p}(v)^T \left\langle \mathbf{p}(v) f \right\rangle_v = w(v) [p_0(v) f_0(x, t) + p_1(v) f_1(x, t) + p_2(v) f_2(x, t)],
\end{align}
where $f_n$ denotes the moment of $f$ with respect to $p_n(v)$:
\begin{align*}
f_n(x, t) = \left\langle p_n(v) f(x, v, t) \right\rangle_v.
\end{align*}
It is easy to show that $\left\langle \mathbf{p} \mathcal{N} \right\rangle_v = \left\langle \mathbf{p} f \right\rangle_v$ using the orthogonality relation \eqref{eqn:p-ortho}.
Because $\mathbf{p}(v)$ and $\bm{\phi}(v)$ are related by the matrix $C$, $\mathcal{N}$ and $f$ share their first three velocity moments:
\begin{align*}
\left\langle \bm{\phi} \mathcal{N} \right\rangle_v = C \langle \mathbf{p}(v) \mathcal{N} \rangle_v = C \left\langle \mathbf{p}(v) f \right\rangle_v = \left\langle \bm{\phi} f \right\rangle_v,
\end{align*}
which immediately implies $\left\langle \bm{\phi} g \right\rangle_v = 0$.
That is, the microscopic part $g$ of the distribution function carries no charge, current, or kinetic energy density.

We now derive equations for the evolution of $\mathcal{N}$ and $g$ as functions of time.
For the macroscopic portion of the distribution function, we are less interested in \( \partial_t \mathcal{N}(x, v, t) \) itself,
and more interested in the evolution of the vector of conserved quantities.
Rather than use the typical equations for charge, current and energy, however, it is more convenient to derive
equations for the moments of \( f \) with respect to \( \mathbf{p}(v) \).
To this end, we apply the operation \( \star \mapsto \left\langle \mathbf{p}(v) \star \right\rangle_v \) to 
the Vlasov equation \eqref{eqn:vlasov-1d1v},
and write the result componentwise with the help of the recurrence relation \eqref{eqn:p-recurrence}
and derivative relations \eqref{eqn:p-derivatives}.
After integrating the $E \cdot \partial_v f$ term by parts, we obtain
\begin{align}
    \label{eqn:first-3}
\partial_t \underbrace{\begin{pmatrix}
f_0 \\ f_1 \\ f_2
\end{pmatrix}}_{U} +
\underbrace{\begin{pmatrix}
b_0 & a_0 \\
a_0 & b_1 & a_1 \\
& a_1 & b_2
\end{pmatrix}}_{V} \Dx \begin{pmatrix}
f_0 \\ f_1 \\ f_2
\end{pmatrix}
+ \underbrace{\begin{pmatrix}
0 \\ 0 \\ a_2 \Dx f_3
\end{pmatrix}}_{a_2 \Dx f_3 \mathbf{e}_2}
+ \qm E \underbrace{\begin{pmatrix}
0 \\
-d_{10} & 0 \\
-d_{20} & -d_{21} & 0
\end{pmatrix}}_{D_v} \begin{pmatrix}
f_0 \\ f_1 \\ f_2
\end{pmatrix} = 
\int_{\Omega_v}
\mathbf{p}(v) Q(f) \, \mathrm{d} v.
\end{align}
We have introduced the name \( U \) for the vector of the first three \( p \)-moments of \( f \), as
well as the symbols \( V \) and \( D_v \) for the flux and velocity derivative matrices, respectively.

Using the fact that \( \mathbf{p}(v) = C^{-1} \bm{\phi}(v) \), we can see that the right-hand side of \eqref{eqn:first-3} is a linear combination of
the moments of \( Q(f) \) with respect to the collision invariants, which by \eqref{eqn:Q-conserves} vanish identically.
Simplifying, we obtain the following initial-boundary value problem for \( U(x, t) \): 
\begin{align}
    \label{eqn:N-ibvp}
\begin{cases}
    \partial_t U(x, t) + V \Dx U + a_2 \Dx f_3 \mathbf{e}_2 + \qm E(x, t) D_v U = 0, & (x, t) \in \Omega_x \times (0, \infty), \\
    U(x, 0) = \left\langle \mathbf{p}(v) f(x, v, 0) \right\rangle_v, & x \in \Omega_x.
\end{cases}
\end{align}
We have derived a system of equations for \( U \), but it is not a closed system.
As expected, a term appears which requires closure: the flux of \( f_2 \) includes a term proportional to \( f_3 \).
This is the appearance in our formulation of the well-known moment closure problem.
The closure information must come from the part of \( f \) that we have projected away, namely \( g \):
\begin{align*}
f_3 = \left\langle p_3(v) f \right\rangle_v = \left\langle p_3(v) g \right\rangle_v,
\end{align*}
where the second equality holds because \( w(v) p_3(v) \) is orthogonal to the subspace \( \Phi \).

The evolution of \( g \) is obtained by substituting \( f = \mathcal{N} + g \) into \eqref{eqn:vlasov-1d1v} and applying the orthogonal projection \( \ppp \):
\begin{align*}
\partial_t \ppp g &= -\partial_t \ppp \mathcal{N} -\ppp \left( v \partial_x \mathcal{N} + \qm E \partial_v \mathcal{N} + v \partial_x g + \qm E \partial_v g \right) + \ppp Q(\mathcal{N} + g) \\
&= -\ppp \left( v \partial_x \mathcal{N} + \qm E \partial_v \mathcal{N} + v \partial_x g + \qm E \partial_v g \right) + \ppp Q(\mathcal{N} + g) \\
&\triangleq \ppp D[E, \mathcal{N}, g].
\end{align*} 
Here, we have interchanged the partial derivative \( \partial_t \) with the time-independent projection \( \ppp \),
and used the fact that \( \ppp \mathcal{N} = 0 \).
Combining this with initial and boundary conditions, \( g \) satisfies the initial-boundary value problem
\begin{align}
    \label{eqn:g-ibvp}
    \begin{cases}
        \partial_t g(x, v, t) = \ppp D[E, \mathcal{N}(x, v, t), g(x, v, t)], & (x, v, t) \in \Omega \times (0, \infty), \\
        g(x, v, 0) = f(x, v, 0) - \mathcal{N}(x, v, 0), & (x, v) \in \Omega, \\
        g(x, v_b, t) = -\mathcal{N}(x, v_b, t) & (x, v, t) \in \Omega_x \times \partial \Omega_v \times (0, \infty).
    \end{cases}
\end{align}
Together, equations \eqref{eqn:N-ibvp} and \eqref{eqn:g-ibvp} constitute an exact macro-micro decomposition of \eqref{eqn:vlasov-1d1v}.
They are coupled on the one hand by the appearance in \eqref{eqn:g-ibvp} of terms involving \( \mathcal{N} \), 
and on the other hand by the gradient of \( f_3 \).
Obtaining the charge, current, and kinetic energy density from the solution to \eqref{eqn:N-ibvp} is trivially accomplished with the transformation matrix \( C \):
\begin{align*}
\begin{pmatrix}
\rho \\ J \\ \kappa
\end{pmatrix}(x, t) = \int_{\Omega_v} \begin{pmatrix}
1 \\ v \\ |v|^2/2
\end{pmatrix} f \,\mathrm{d} v = C U(x, t).
\end{align*}
We also expand $\rho$, \( u \), and \( T \), defined by \eqref{eqn:uT-def}, in terms of \( C \) and \( U \).
\begin{align}
    \label{eqn:uT-formulas}
    \rho &= c_{00} f_0, \\
u &= \frac{1}{\rho} \int vf \,\mathrm{d} v = \frac{c_{11} f_1 + c_{10} f_0}{\rho}, \\
T &= \frac{1}{\rho} \int (v - u)^2 f \,\mathrm{d} v = \frac{1}{\rho} \left( \int v^2f\,\mathrm{d} v - \rho u^2 \right) 
= \frac{2c_{22} f_2 + 2c_{21} f_1 + 2c_{20} f_0}{\rho} - u^2.
\end{align}

\section{Time discretization}
\label{sec:time-discretization}
We turn now to a discussion of how the coupled initial-boundary value problems \eqref{eqn:N-ibvp} and \eqref{eqn:g-ibvp} are discretized in time.
Our strategy is to discretize \eqref{eqn:N-ibvp} using standard conservative techniques, while \eqref{eqn:g-ibvp} is discretized
using the projector-splitting dynamical low-rank integrator \cite{lubichProjectorsplittingIntegratorDynamical2014}.
Because the Vlasov equation must be coupled with an equation to advance the electric field,
developing a scheme that is conservative overall is rather involved, and requires a careful consideration
of the way that the conservation properties transfer from the continuous to the discrete level.
This section is organized as follows.

\begin{itemize}
    \item Section \ref{sec:g-dlra} presents the projector-splitting integrator and equations of motion for the low-rank factors of $g$.
    \item In Section \ref{sec:first-order-integrator} we present a first-order integrator which can achieve conservation of charge density, and either current or energy density depending how $f$ is coupled to the electric field.
    \item In Section \ref{sec:second-order-integrator} we present a second-order integrator which can achieve conservation of charge and energy density by coupling with \Ampere's law.
    \item Section \ref{sec:low-rank-factors} gives the details of calculations necessary to implement the evolution equations for the low-rank factors of $g$.
\end{itemize}

\subsection{Dynamical low-rank approximation of \( g \)}
\label{sec:g-dlra}
We now describe the dynamical low-rank approximation for \( g \), and write down
the equations of motion of the low-rank factors.
One of our main contributions in this work is the ability to combine a conservative scheme
with any dynamical low-rank integrator: the DLR equations of motion are unchanged 
by our scheme, and conservation does not rely on a rank augmentation at any intermediate step.
We choose the projector-splitting integrator for the ease with which it may
be formally extended to second-order accuracy via a Strang splitting scheme
with no increase of the rank in intermediate steps.
However, no proof of robust second-order accuracy exists.
In a very recent development, a robust second-order BUG integrator has been
introduced in \cite{cerutiRobustSecondorderLowrank2024},
with which our macro-micro decomposition should also be compatible.
Prior work has demonstrated locally conservative methods for the Vlasov equation with the traditional \cite{einkemmerMassMomentumEnergy2021} and robust basis-update \& Galerkin \cite{einkemmerRobustConservativeDynamical2023} dynamical low-rank integrators.
For the projector-splitting integrator, only a globally conservative method based on Lagrange multipliers 
\cite{einkemmerQuasiConservativeDynamicalLowRank2019} has been demonstrated previously.
To our knowledge this is the first locally conservative scheme in the projector-splitting integrator.

Our approach is to apply the standard projector-splitting integrator, but to use the weighted inner product \( \left\langle \cdot \right\rangle_\Lwinv \)
for projection along \( v \).
The low-rank ansatz for \( g \) is
\begin{align}
    \label{eqn:g-ansatz}
    g(x, v, t) = \sum_{ij} X_i(x, t) S_{ij}(t) V_j(v, t), 
\end{align}
with the basis functions \( X_i(x, t) \) satisfying
\begin{align}
X_i \in \left\{ X_i \in L^2(\Omega_x) : \left\langle X_i, X_k \right\rangle_x = \delta_{ik} \right\},
\end{align}
and the velocity basis functions \( V_j(v, t) \) satisfying
\begin{align}
V_j \in \left\{ V_j \in L^2(\Omega_v, w^{-1}(v)) : \left\langle V_j, V_l \right\rangle_\Lwinv = \delta_{jl} \right\}.
\end{align}
While the basis functions in \( x \) are chosen the same way as in the unmodified dynamical low-rank approximation (c.f. \cite{einkemmerLowRankProjectorSplittingIntegrator2018}),
the \( v \) basis functions are chosen to be orthonormal with respect to the \( w^{-1} \)-weighted inner product.

The equations of motion for the low-rank factors are obtained by projecting \eqref{eqn:g-ibvp} onto the subspaces spanned by $X_i$ and $V_j$.
The low-rank projection is of the form
\begin{align*}
    \partial_t g &= \sum_j \left\langle V_j, \ppp D[E, \mathcal{N}, g] \right\rangle_\Lwinv V_j - \sum_{ij} X_i \left\langle X_i V_j, \ppp D[E, \mathcal{N}, g] \right\rangle_\xLwinv V_j \\
    &\quad + \sum_i X_i \left\langle X_i, \ppp D[E, \mathcal{N}, g] \right\rangle_x.
\end{align*}
Defining $K_j$ and $L_i$ via
\begin{align}
K_j(x, t) = \sum_{i} X_i(x, t) S_{ij}(t), \qquad L_i(v, t) = \sum_{j} S_{ij}(t) V_j(v, t),
\end{align}
we may write
\begin{align}
    \label{eqn:dt-g-splittable}
    \partial_t g = \sum_j \partial_t K_j V_j + \sum_{ij} X_i \partial_t S_{ij} V_j + \sum_i X_i \partial_t L_i,
\end{align}
with
\begin{align}
    \label{eqn:K-eqn}
    \partial_t K_j &= \left\langle V_j, \ppp D[E, \mathcal{N}, g] \right\rangle_\Lwinv, \\
    \label{eqn:S-eqn}
    \partial_t S_{ij} &= -\left\langle X_i V_j, \ppp D[E, \mathcal{N}, g] \right\rangle_\xLwinv ,\\
    \label{eqn:L-eqn}
    \partial_t L_i &= \left\langle X_i, \ppp D[E, \mathcal{N}, g] \right\rangle_x.
\end{align}
We will initialize the low-rank factors by performing a singular value decomposition of the discretization of \( g \)
and truncating to rank \( r \).\footnote{
    Performing a full SVD of $g$ is feasible for the 1D1V problems we consider here, but it may be necessary
    to avoid this for higher-dimensional problems.
    Randomized algorithms such as the randomized SVD can be used to compute highly accurate decompositions of
    enormous matrices at a small fraction of the cost.
    See \cite{martinssonRandomizedNumericalLinear2020} for an excellent overview of these ideas.
}
Note that at this point no discretization in time nor operator splitting has been performed.
In the next section we describe how the standard first-order splitting of \eqref{eqn:dt-g-splittable} is
incorporated into a first-order in time discretization for our macro-micro decomposition.

\subsection{First-order integrator}
\label{sec:first-order-integrator}

In this section we describe a first-order in time integrator which conserves charge and either current or energy exactly.
The algorithm computes the following time advance
between times \( t^n \) and \( t^{n+1} = t^n + \Delta t \):
\begin{align*}
\begin{pmatrix}
E^{n} \\ f_0^{n} \\ f_1^{n} \\ f_2^{n} \\ X^{n} \\ S^{n} \\ V^{n}
\end{pmatrix}
\mapsto
\begin{pmatrix}
E^{n+1} \\ f_0^{n+1} \\ f_1^{n+1} \\ f_2^{n+1} \\ X^{n+1} \\ S^{n+1} \\ V^{n+1}
\end{pmatrix}
\end{align*}

\begin{enumerate}
    \item \textbf{Calculate the electric field:}
        The choice of electric field solution depends on whether current or energy conservation is desired.
        \begin{enumerate}
            \item \textbf{For current conservation}: 
                Determine the electric field from Gauss's law.
                Solve the following Poisson equation for the electric potential $\varphi^{n}$:
                \begin{align}
                    \label{eqn:gauss_law_poisson}
                    \partial_x^2 \varphi^{n} = -(\rho^{n} - \rho_0), \quad \rho^{n} = c_{00} f_0^{n}.
                \end{align}
                Then the electric field $E^*$ to be used is
                \begin{align}
                    \label{eqn:gauss-dt}
                    E^* = E^n = -\partial_x \varphi^n.
                \end{align}

        \item \textbf{For energy conservation}: Perform a single Forward Euler step of \Ampere's law \eqref{eqn:ampere} to obtain $E^{n+1}$:
        \begin{align}
            \label{eqn:ampere-fe}
        \frac{E^{n+1} - E^n}{\Delta t} = -J^n,
        \end{align}
        where \( J^n = \int_{\Omega_v} vf^n\,\mathrm{d} v = c_{10} f_0^n + c_{11} f_1^n \).
        For the current timestep, use a time-centered electric field:
        \begin{align}
        \label{eqn:E-centered}
        E^* = \frac{E^{n+1} + E^n}{2}.
        \end{align}
        \end{enumerate}

    \item \textbf{Advance conserved quantities:} Perform a single Forward Euler timestep of \eqref{eqn:first-3}:
    \begin{align}
        \label{eqn:step-f0}
        \frac{f_0^{n+1} - f_0^n}{\Delta t} &= \Dx (-b_0 f_0^n - a_0 f_1^n), \\
        \label{eqn:step-f1}
    \frac{f_1^{n+1} - f_1^n}{\Delta t} &= -a_1 \partial_x f_2^n - b_1 \partial_x f_1^n - a_0 \partial_x f_0^n + \qm E^* d_{10} f_0^n, \\
    \label{eqn:step-f2}
    \frac{f_2^{n+1} - f_2^n}{\Delta t} &= -a_2 \partial_x f_3^n - b_2 \partial_x f_2^n - a_1 \partial_x f_1^n + \qm E^* (d_{20} f_0^n + d_{21} f_1^n).
    \end{align}
    The third moment \( f_3^n \) can be computed from the low-rank factors like so:
    \begin{align}
        \label{eqn:f-3}
    f_3^n = \sum_{ij} X^n_i S^n_{ij} \left\langle p_3(v) V^n_j \right\rangle_v.
    \end{align}
    Define \( \mathcal{N}^n(x, v) \) by
    \begin{align*}
    \mathcal{N}^n(x, v) = w(v) [p_0(v) f_0^n(x) + p_1(v) f_1^n(x) + p_2(v) f_2^n(x)].
    \end{align*}
    
    \item \textbf{K step:} Calculate \( K_j^{n}(x) = \sum_i X_i^n(x) S^n_{ij} \), and perform a Forward Euler step of \eqref{eqn:K-eqn} to obtain \( K_j^{n+1}(x) \):
    \begin{align}
    \label{eqn:K-step}
    \frac{K_j^{n+1} - K_j^n}{\Delta t} = \left\langle V^n_j, \ppp D\left[E^*, \mathcal{N}^n, \sum_l K_l^n V_l^n \right] \right\rangle_\Lwinv.
    \end{align}
    Perform a QR decomposition of \( K_j^{n+1}(x) \) to obtain \( X^{n+1}_i(x) \) and \( S_{ij}^\prime \).

    \item \textbf{S step:} Perform a Forward Euler step of \eqref{eqn:S-eqn} to obtain \( S_{ij}^{\prime\prime} \):
    \begin{align}
    \label{eqn:S-step}
    \frac{S_{ij}^{\prime\prime} - S_{ij}^\prime}{\Delta t} = -\left\langle X_i^{n+1} V_j^n, \ppp D\left[E^*, \mathcal{N}^n, \sum_{kl} X_k^{n+1} S_{kl}^\prime V_l^n \right] \right\rangle_\xLwinv.
    \end{align}

    \item \textbf{L step:} Calculate \( L_i^{n}(v) = \sum_j S^{\prime\prime}_{ij} V_j^n(v) \), and perform a Forward Euler step of \eqref{eqn:L-eqn} to obtain \( L_i^{n+1}(v) \):
    \begin{align}
    \label{eqn:L-step}
    \frac{L_i^{n+1} - L_i^n}{\Delta t} = \left\langle X^{n+1}_i, \ppp D\left[E^*, \mathcal{N}^n, \sum_{k} X_k^{n+1} L_k^n \right] \right\rangle_x.
    \end{align}
        At this point, we must perform a QR decomposition of $L^{n+1}_i(v)$ to obtain
        the new velocity basis $V_j^{n+1}$ and singular value matrix $S_{ij}^{n+1}(v)$.
        However, we also require that the velocity basis so obtained is orthogonal to $\Phi$.

        We accomplish this by prepending the three functions $w(v) \mathbf{p}(v)$ to the vector of functions $L_i^{n+1}(v)$.
        Denoting vector concatenation by square brackets, we compute a QR decomposition
        \begin{align}
            \label{eqn:V_jnplus1_QR}
            [ w(v) \mathbf{p}(v) \enspace V_j^{n+1}(v) ] R = [w(v) \mathbf{p}(v) \enspace L_i^{n+1}(v)],
        \end{align}
        with respect to the function inner product $\left\langle \cdot \right\rangle_{w(v)^{-1}}$.
        The QR decomposition leaves the first three functions unchanged since they are already orthogonal,
        and guarantees that the resulting basis $V_j^{n+1}(v)$ is orthogonal to $\Phi$, which is spanned by $w(v) \mathbf{p}(v)$.

        The updated matrix of singular values is then given by the trailing $r \times r$ minor of $R$:
        \begin{align*}
            S^{n+1}_{ij} = R_{i+3,j+3}.
        \end{align*}

\end{enumerate}

\subsubsection{Proof of conservation}

    We now prove local conservation of our scheme.
The local conservation statement is satisfied by conservative dynamical low-rank integrators such as \cite{einkemmerRobustConservativeDynamical2023}
and \cite{einkemmerConservationPropertiesAugmented2023}, and is important for ensuring that the solution
has the properties of a hyperbolic conservation law, such as finite wavespeeds.
We formulate this property in the following theorem:
\begin{thm}
    \label{thm:local-conservation}
    Define the conserved quantities of charge, current, and kinetic energy density at time level $t^n$ as follows:
    \begin{align}
        \label{eqn:def-rho-n}
        \rho^n &= \left\langle f^n, 1 \right\rangle_v, \\
        \label{eqn:def-J-n}
        J^n &= \left\langle f^n, v \right\rangle_v, \\
        \label{eqn:def-kappa-n}
        \kappa^n &= \left\langle f^n, \frac{1}{2} v^2 \right\rangle_v.
    \end{align}
    The first-order integration algorithm of Section \ref{sec:first-order-integrator} satisfies three
    local source-balance laws
    \begin{align}
        \label{eqn:local_rho_conservation}
        & \frac{\rho^{n+1} - \rho^n}{\Delta t} + \partial_x J^n = 0, \\
        \label{eqn:local_J_conservation}
        & \frac{J^{n+1} - J^n}{\Delta t} + 2\partial_x \kappa^n = \rho^n E^*, \\
        \label{eqn:local_kappa_conservation}
        & \frac{\kappa^{n+1} - \kappa^n}{\Delta t} + \partial_x \left\langle f^n, \frac{v^3}{2} \right\rangle_v = J^n E^*.
    \end{align}
\end{thm}
\begin{proof}
    By \eqref{eqn:V_jnplus1_QR}, $P_\Phi V_j^n = 0$, so $\left\langle g^n, \bm{\phi}(v) \right\rangle_v = 0$ for all times $n$.
    Therefore the only contribution to the conserved quantities comes from $f_0$, $f_1$, and $f_2$.

    To show mass conservation, we make use of \eqref{eqn:phi_p_relation_c}, \eqref{eqn:step-f0},
    and \eqref{eqn:p-recurrence} to obtain
    \begin{align*}
        \frac{\rho^{n+1} - \rho^n}{\Delta t} = c_{00} \frac{f_0^{n+1} - f_0^n}{\Delta t} &= c_{00} \Dx (-b_0 f_0^n - a_0 f_1^n) \\
                                                                                                                     &= -c_{00} \Dx \left\langle f^n, v p_0(v) \right\rangle_v \\
                                                                                                                     &= - \Dx \left\langle f^n, v \right\rangle_v \\
                                                                                                                     &= -\Dx J^n.
    \end{align*}

    For current conservation, we use \eqref{eqn:step-f1} and \eqref{eqn:phi_p_relation_c}:
    \begin{align*}
        \frac{J^{n+1} - J^n}{\Delta t} &= c_{11} \frac{f_1^{n+1} - f_1^n}{\Delta t} + c_{10} \frac{f_0^{n+1} - f_0^n}{\Delta t} \\
                                                                         &= c_{11}(-a_1 \partial_x f_2^n - b_1 \partial_x f_1^n - a_0 \partial_x f_0^n + E^* d_{10} f_0^n) + c_{10}(-b_0 \partial_x f_0^n - a_0 \partial_x f_1^n) \\
                                                                         &= -c_{11} \partial_x \left\langle f^n, v p_1(v) \right\rangle_v - c_{10} \partial_x \left\langle f^n, v p_0(v) \right\rangle_v + c_{11} E^* d_{10} \left\langle f^n, p_0(v) \right\rangle_v \\
                                                                         &= -\partial_x \left\langle f^n, v^2 \right\rangle_v + c_{11} E^* \left\langle f^n, p_1'(v) \right\rangle_v \\
                                                                         &= -\partial_x \left\langle f^n, v^2 \right\rangle_v - E^* \left\langle \partial_v f^n, v \right\rangle_v \\
                                                                         &= - 2\partial_x \kappa^n + E^* \rho^n.
    \end{align*}

    For kinetic energy conservation, using \eqref{eqn:step-f2} we obtain
    \begin{align*}
        2\frac{\kappa^{n+1} - \kappa^n}{\Delta t} &= c_{22} \frac{f_2^{n+1} - f_2^n}{\Delta t} + c_{21} \frac{f_1^{n+1} - f_1^n}{\Delta t} + c_{20} \frac{f_0^{n+1} - f_0^n}{\Delta t} \\
                                                                           &= -\partial_x \left\langle f^n, v [c_{22} p_2(v) + c_{21} p_1(v) + c_{20} p_0(v)] \right\rangle_v \\
                                                                           &\quad + E^* \left\langle f^n,  [c_{22} p_2'(v) + c_{21} p_1'(v)] \right\rangle_v \\
                                                                           &= -\partial_x \left\langle f^n, v^3 \right\rangle_v - E^* \left\langle \partial_v f^n, v^2 \right\rangle_v.
    \end{align*}
    Simplifying slightly,
    \begin{align*}
        \frac{\kappa^{n+1} - \kappa^n}{\Delta t} + \partial_x \left\langle f^n, \frac{v^3}{2} \right\rangle_v = J^n E^*.
    \end{align*}
    This completes the proof.
\end{proof}

Theorem \ref{thm:local-conservation} shows that the macroscopic current and kinetic energy satisfy the same source-balance laws as the full kinetic equation.
The source terms in \eqref{eqn:local_J_conservation} and \eqref{eqn:local_kappa_conservation} reflect
the fact that the particle distribution function exchanges momentum and energy with the electric field.
Depending on the choice of electric field solve, we can show that the total energy, including electric field energy, satisfies a local conservation law:
\begin{corollary}
    Define the total energy at time level $t^n$ as
    \begin{align}
        \label{eqn:def_e_n}
    e^n = \kappa^n + \frac{|E^n|^2}{2}.
    \end{align}
    Then the first-order integration algorithm of Section \ref{sec:first-order-integrator} with the choice of \Ampere solve and $E^* = \frac{E^{n+1} + E^n}{2}$ satisfies a local conservation law,
    \begin{align}
        \label{eqn:local_e_conservation}
        \frac{e^{n+1} - e^n}{\Delta t} + \partial_x \left\langle f^n, \frac{v^3}{2} \right\rangle_v = 0.
    \end{align}
\end{corollary}
\begin{proof}
    The proof is a simple application of \eqref{eqn:local_kappa_conservation} and \eqref{eqn:ampere-fe}:
    \begin{align*}
        \frac{e^{n+1} - e^n}{\Delta t} &= \frac{\kappa^{n+1} - \kappa^n}{\Delta t} + \frac{|E^{n+1}|^2 - |E^n|^2}{2\Delta t}  \\
                                       &= -\partial_x \left\langle f^n, \frac{v^3}{2} \right\rangle_v + \frac{E^{n+1} + E^n}{2} J^n + \frac{E^{n+1} + E^n}{2} \frac{E^{n+1} - E^n}{\Delta t} \\
                                       &= -\partial_x \left\langle f^n, \frac{v^3}{2} \right\rangle_v.
    \end{align*}
\end{proof}
To summarize, \eqref{eqn:local_rho_conservation} demonstrates exact local conservation of charge,
and \eqref{eqn:local_e_conservation} demonstrates exact local conservation of total energy
for the \Ampere's law variation of the first-order integrator.
Per \eqref{eqn:local_J_conservation}, current is not locally conserved, since the mobile particles ``push against'' the background
charge density $\rho_0$.
However, we can prove a global conservation statement for current in the special case of periodic
boundary conditions and the Gauss's law variation of the first-order integrator:
\begin{corollary}
If Gauss's law and the uncentered electric field $E^* = E^n$ are chosen in the first-order integrator, then
\begin{align*}
    \int_{\Omega_x} \frac{J^{n+1} - J^n}{\Delta t}\,\mathrm{d} x = 0.
\end{align*}
\end{corollary}
\begin{proof}
    Integrate \eqref{eqn:local_J_conservation} over the periodic domain $\Omega_x$ to obtain
    \begin{align*}
        \int_{\Omega_x} \frac{J^{n+1} - J^n}{\Delta t}\,\mathrm{d} x = \int_{\Omega_x} \rho^n E^* \, \mathrm{d} x = -\int_{\Omega_x} \rho^n \partial_x \varphi^n \, \mathrm{d} x = \int_{\Omega_x} (\partial_x^2 \varphi^n - \rho_0) \partial_x \varphi^n \, \mathrm{d} x = 0.
    \end{align*}
    We have used \eqref{eqn:gauss_law_poisson} and \eqref{eqn:gauss-dt}, and eliminated all
    integrals of total derivatives.
\end{proof}

\begin{remark}
    A fully discrete conservative scheme for the first-order integrator of Section \ref{sec:first-order-integrator}
    is easily achieved by using a conservative spatial discretization for the $\partial_x$ operators
    appearing in equations \eqref{eqn:step-f0}-\eqref{eqn:step-f2}.
    Such a scheme will satisfy exact conservation of charge, and conservation of energy if \Ampere's law is used.
    Furthermore, a discrete scheme using Gauss's law will exactly conserve current on a periodic
    domain if the discrete $\partial_x$ operators in \eqref{eqn:gauss_law_poisson} and \eqref{eqn:gauss-dt}
    satisfy a summation-by-parts identity.
    One such discretization is the common second-order centered finite difference scheme.
\end{remark}

\subsection{Second-order time integrator}
\label{sec:second-order-integrator}

The projector-splitting framework may be formally extended to second-order accuracy
by using a Strang splitting of \eqref{eqn:dt-g-splittable}.
A Strang splitting of the dynamical low-rank projection has been used in \cite{einkemmerLowRankProjectorSplittingIntegrator2018} and \cite{einkemmerAcceleratingSimulationKinetic2023}
to obtain second-order accurate solutions to certain problems.
We stress that it has not been proven that the Strang splitting of the DLR projection
is robust to vanishing singular values, as it has for the first-order splitting \cite{kieriDiscretizedDynamicalLowRank2016}.
That is, no robust second-order accuracy result exists.
Nevertheless, the second-order Strang splitting scheme shows significant practical benefits
on plasma problems, as demonstrated for example in \cite{einkemmerAcceleratingSimulationKinetic2023}.
To achieve overall second-order accuracy, some care is required when coupling the DLR scheme with the time
splitting for $\mathcal{N}$ and $E$.
Here we present one such scheme and prove that it exactly conserves energy.

\begin{enumerate}
    \item \textbf{Half step of conserved quantities}: $(f_0^n, f_1^n, f_2^n) \mapsto (f_0^{n+1/2}, f_1^{n+1/2}, f_2^{n+1/2})$ using $g^n, E^n$:
        \begin{align}
            \label{eqn:step-f0-2}
        \frac{f_0^{n+1/2} - f_0^n}{\Delta t / 2} &= \partial_x (-b_0 f_0^n - a_0 f_1^n), \\
            \label{eqn:step-f1-2}
        \frac{f_1^{n+1/2} - f_1^n}{\Delta t / 2} &= \partial_x (-a_1 f_2^n - b_1 f_1^n - a_0 f_0^n) + E^n d_{10}f_0^n, \\
            \label{eqn:step-f2-2}
        \frac{f_2^{n+1/2} - f_2^n}{\Delta t / 2} &= \partial_x (-a_2 f_3^n - b_2 f_2^n - a_1 f_1^n) + E^n (d_{20} f_0^n + d_{21} f_1^n),
        \end{align}
        where $f_3^n$ is defined as in \eqref{eqn:f-3}.
        Define
        \begin{align*}
            \mathcal{N}^{n+1/2} = w(v) [p_0(v) f_0^{n+1/2} + p_1(v) f_1^{n+1/2} + p_2(v) f_2^{n+1/2}].
        \end{align*}

    \item \textbf{\Ampere\;solve}: $E_n \mapsto E^{n+1}$ using $f_0^{n+1/2}, f_1^{n+1/2}$:
        \begin{align}
            \label{eqn:ampere-2nd-order}
            \frac{E^{n+1} - E^n}{\Delta t} = -J^{n+1/2}, \quad J^{n+1/2} = c_{10}f_0^{n+1/2} + c_{11} f_1^{n+1/2}.
        \end{align}
        Define $E^{n+1/2} = \frac{E^{n+1} + E^n}{2}$.

    \item \textbf{K step}: $(X^n, S^n, V^n) \mapsto (X^{n+1/2}, S^1, V^n)$ using $\mathcal{N}^{n+1/2}, E^{n+1/2}$.

    \item \textbf{S step}: $(X^{n+1/2}, S^1, V^n) \mapsto (X^{n+1/2}, S^2, V^n)$ using $\mathcal{N}^{n+1/2}, E^{n+1/2}$.

    \item \textbf{L step}: $(X^{n+1/2}, S^2, V^n) \mapsto (X^{n+1/2}, S^{n+1/2}, V^{n+1/2})$ using $\mathcal{N}^{n+1/2}, E^{n+1/2}$.

\item \textbf{L step}: $(X^{n+1/2}, S^{n+1/2}, V^{n+1/2}) \mapsto (X^{n+1/2}, S^3, V^{n+1})$ using $\mathcal{N}^{n+1/2}, E^{n+1/2}$.

\item \textbf{S step}: $(X^{n+1/2}, S^3, V^{n+1}) \mapsto (X^{n+1/2} S^4, V^{n+1})$ using $\mathcal{N}^{n+1/2}, E^{n+1/2}$.

\item \textbf{K step}: $(X^{n+1/2}, S^4, V^{n+1}) \mapsto (X^{n+1}, S^{n+1}, V^{n+1})$ using $\mathcal{N}^{n+1/2}, E^{n+1/2}$.

\item $(f_0^n, f_1^n, f_2^n) \mapsto (f_0^{n+1}, f_1^{n+1}, f_2^{n+1})$ using $g^{n+1/2}, E^{n+1/2},$ and $(f_0, f_1, f_2)^{n+1/2}$:
        \begin{align}
            \label{eqn:f0-2nd-order-step}
            \frac{f_0^{n+1} - f_0^n}{\Delta t} &= \partial_x (-b_0 f_0^{n+1/2} - a_0 f_1^{n+1/2}), \\
            \label{eqn:f1-2nd-order-step}
            \frac{f_1^{n+1} - f_1^n}{\Delta t} &= \partial_x (-a_1 f_2^{n+1/2} - b_1 f_1^{n+1/2} - a_0 f_0^{n+1/2}) + E^{n+1/2} d_{10}f_0^{n+1/2}, \\
            \label{eqn:f2-2nd-order-step}
            \frac{f_2^{n+1} - f_2^n}{\Delta t} &= \partial_x (-a_2 f_3^{n+1/2} - b_2 f_2^{n+1/2} - a_1 f_1^{n+1/2}) + E^{n+1/2} (d_{20} f_0^{n+1/2} + d_{21} f_1^{n+1/2}),
        \end{align}
        where
        \begin{align*}
            f_3^{n+1/2} = \sum_{ij} X_i^{n+1/2} S_{ij}^{n+1/2} \left\langle p_3(v) V_j^{n+1/2} \right\rangle_v.
        \end{align*}
\end{enumerate}
Note that in step 9 we make use of $g^{n+1/2}$, which is defined in terms of $(X^{n+1/2}, S^{n+1/2}, V^{n+1/2})$.
This means that we cannot easily combine steps 5 and 6 into a single substep of size $\Delta t$ as is common in
Strang splitting schemes.
To achieve overall second-order accuracy in time, each of the low-rank factor steps 3-8 must be
accomplished using a standard time integration scheme that is at least second-order accurate.
We choose the SSPRK2 scheme \cite{gottliebStrongStabilityPreserving2011}, 
which for an autonomous ordinary differential equation $q'(t) = F(q)$ is
\begin{align}
    q^* &= q^n + \Delta t F(q^n) \\
    q^{n+1} &= \frac{q^n}{2} + \frac{1}{2} \left( q^* + \Delta t F(q^*) \right). \nonumber 
\end{align}

\subsubsection{Proof of charge and energy conservation}
\begin{thm}
    \label{thm:conservation2}
    The second-order integrator satisfies local conservation of charge and total energy:
    \begin{align}
        \label{eqn:local_rho_conservation_2nd_order}
        &\frac{\rho^{n+1} - \rho^n}{\Delta t} + \partial_x J^{n+1/2} = 0, \\
        \label{eqn:local_e_conservation_2nd_order}
        &\frac{e^{n+1} - e^n}{\Delta t} + \partial_x \left\langle \frac{v^3}{2}, f^{n+1/2} \right\rangle_v = 0,
    \end{align}
    where $\rho^n$ and $e^n$ are defined as in \eqref{eqn:def-rho-n} and \eqref{eqn:def_e_n} respectively,
    and $J^{n+1/2}$ is defined as in \eqref{eqn:def-J-n} but at $t^{n+1/2}$.
\end{thm}
\begin{proof}
    By \eqref{eqn:V_jnplus1_QR}, $P_\Phi V_j^n = \mathbf{0}$, so $\left\langle g^n, \bm{\phi}(v) \right\rangle_v = 0$ for all times $n$.
    Therefore, the only contributions to the charge and total energy come from $f_0, f_1, f_2$ and $E$.
    From equations \eqref{eqn:f0-2nd-order-step} and \eqref{eqn:f2-2nd-order-step}, by
    following the proof of Theorem \ref{thm:local-conservation} with fluxes and source terms evaluated
    at $t^{n+1/2}$, we obtain
    \begin{align}
        &\frac{\rho^{n+1} - \rho^n}{\Delta t} + \partial_x J^{n+1/2} = 0, \\
        \label{eqn:kappa_conservation_2nd_order}
        &\frac{\kappa^{n+1} - \kappa^n}{\Delta t} + \partial_x \left\langle \frac{v^3}{2}, f^{n+1/2} \right\rangle_v = J^{n+1/2} E^{n+1/2}.
    \end{align}
    Combining \eqref{eqn:kappa_conservation_2nd_order} with \eqref{eqn:ampere-2nd-order} and the definition of $E^{n+1/2}$, we derive the stated total energy conservation law \eqref{eqn:local_e_conservation_2nd_order}.
\end{proof}

\begin{corollary}
    A fully discrete scheme for the second-order integrator of Section \ref{sec:second-order-integrator},
    which uses a conservative spatial discretization for the $\partial_x$ operators
    appearing in equations \eqref{eqn:step-f0-2}-\eqref{eqn:step-f2-2} and \eqref{eqn:f0-2nd-order-step}-\eqref{eqn:f2-2nd-order-step} will satisfy exact discrete
    charge and energy conservation.
\end{corollary}

\subsection{Substeps for low-rank factors}
\label{sec:low-rank-factors}

In this section we expand each of the low-rank factors' equation of motion, \eqref{eqn:K-eqn}, \eqref{eqn:S-eqn}, \eqref{eqn:L-eqn}.
Our purpose is to demonstrate that the proposed algorithm is efficient in the sense of not requiring
operations that have a computational cost of $\mathcal{O}(N_x N_v)$, where $N_x, N_v$ represent the degrees of freedom
used to discretize $x, v$ respectively.
To avoid proliferation of indices, in this section we consider only the case of simple Forward Euler steps
used in the first-order integrator.
The computational cost of the second-order integrator, whose substeps are themselves SSPRK2 steps,
differs by only a constant factor.

The terms on the right-hand side of the K and S steps can be simplified by the following observation.
\begin{prop}
    \label{prop:Vj_projection}
    Let the basis functions $V_j^n$ be evolved according to the first-order time integrator of Section \ref{sec:first-order-integrator}.
    Then, for all \( h(v) \in L^2(\Omega_v, w^{-1}(v)) \),
    the following identity holds for all \( V_j^n(v) \):
    \begin{align}
        \left\langle V_j^n(v), \ppp h(v) \right\rangle_{w^{-1}(v)} = \left\langle V_j^n(v), h(v) \right\rangle_{w^{-1}(v)}.
    \end{align}
\end{prop}
\begin{proof}
    Decompose \( h(v) \) as \( h(v) = P_\Phi h(v) + \ppp h(v) \), and use \( \left\langle V_j^n(v), P_\Phi h(v) \right\rangle_{w^{-1}(v)} = 0 \),
    since \( V_j^n \) is orthogonal to the range of \( P_\Phi \).
\end{proof}

\subsubsection*{Applying boundary conditions in \( v \)}
For a bounded velocity domain such as the finite difference discretization we will discuss below,
we must apply a boundary condition in $v$.
This is clear in the case of the L step, where the boundary condition is applied to a hyperbolic term.
It is also true for the K and S steps, where the boundary condition is required for the evaluation of a derivative
under an integral.
In \cite{huAdaptiveDynamicalLow2022} the authors show how to treat boundary conditions in \( x \) in the dynamical low-rank framework;
we use the same approach.
The idea is that the Dirichlet boundary condition \( g(x, v_b, t) = -\mathcal{N}(x, v_b, t) \) does not induce a boundary condition on the basis functions \( V_j \) directly.
Rather, by projecting onto \( X_i \), we can see that the boundary condition should be applied to the weighted function basis \( L_i \):
\begin{align*}
\left\langle X_i, g(x, v_b, t) \right\rangle_x = L_i(v_b, t) = \left\langle X_i, -\mathcal{N}(x, v_b, t) \right\rangle_x.
\end{align*}
Then, when expanding the low-rank equations of motion, we interpret terms involving \( \partial_v g \) as \( \sum_k X_k \partial_v L_k \), rather than the usual
\( \sum_{kl} X_k S_{kl} \partial_v V_l \), with the plan of applying the boundary condition on \( L_k \) in the course of evaluating the derivative.

Our assumption of periodic boundary conditions in \( x \) saves us the trouble of performing the same transformation to terms involving \( \partial_x g \);
however this poses no essential difficulty, and nontrivial nonperiodic boundary conditions in \( x \) are a straightforward extension of this scheme.

\subsubsection*{K step}
We may expand \eqref{eqn:K-step} by substituting \eqref{eqn:g-ansatz} into \eqref{eqn:K-eqn}.
Since the collisional moments $u$ and $T$ are held constant during the K step, at a time level $t^n$, we may write
\begin{align*}
Q(\mathcal{N}^n + g) = Q^n(\mathcal{N}^n + g) = Q^n(\mathcal{N}^n) + Q^n(g),
\end{align*}
where
\begin{align*}
Q^n(f) = \nu \partial_v (T^n \partial_v f + (v - u^n) f).
\end{align*}
After using Proposition \ref{prop:Vj_projection} to eliminate appearances of \( \ppp \),
this gives
\begin{align}
    \label{eqn:K-step-2}
\frac{K^{n+1}_j - K^n_j}{\Delta t} &= \left\langle V_j^n, D\left[E^*, \mathcal{N}^n, \sum_l K_l^n V_l^n \right] \right\rangle_\Lwinv \\
&= -\left\langle V_j^n, v \Dx \mathcal{N}^n + \qm E^* \Dv \mathcal{N}^n \right\rangle_\Lwinv + \left\langle V_j^n, Q^n(\mathcal{N}^n) \right\rangle_\Lwinv \nonumber \\
&\quad - \sum_{l} \left\langle V_j^n, v V_l^n \right\rangle_\Lwinv \Dx K^n_l  - \qm \sum_{k} X^n_k \left\langle V^n_j, \Dv L^n_k \right\rangle_\Lwinv E^* \\
&\quad + \nu \sum_{k} X^n_k \left[T^n \left\langle V^n_j, \Dv^2 L^n_k \right\rangle_\Lwinv + \left\langle V^n_j, \Dv (vL^n_k) \right\rangle_\Lwinv - u^n \left\langle V^n_j, \Dv L^n_k \right\rangle_\Lwinv \right], \nonumber
\end{align}
where \( L_i^n = \sum_j S_{ij}^n V_j^n \).
Derivative terms involving \( L_i^n \) are to be calculated using the projected boundary conditions on \( L \):
\begin{align*}
L^n_i(v_b) = \left\langle X_i^n, -\mathcal{N}^n(x, v_b) \right\rangle_x.
\end{align*}
Because some of the inner products appearing in \eqref{eqn:K-step-2} will also appear in the S step, we can save some computational effort by
computing and saving the vectors and matrices in \eqref{eqn:K-step-2}: 
\begin{align}
    \label{eqn:y1-def}
    \bm{y}^1_j = \left\langle V_j^n, v \Dx \mathcal{N}^n \right\rangle_\Lwinv, 
    \quad \bm{y}^2_j = \left\langle V_j^n, E^* \Dv \mathcal{N}^n \right\rangle_\Lwinv, 
    \quad \bm{y}^3_j(x) = \left\langle V_j^n, Q^n(\mathcal{N}^n) \right\rangle_\Lwinv
\end{align}
\begin{align}
    \label{eqn:Ajl-def}
    \bm{A}_{jl} = \left\langle V_j^n, v V_l^n \right\rangle_\Lwinv, 
    \quad \bm{D}^1_{jk} = \left\langle V^n_j, \partial_v L^n_k \right\rangle_\Lwinv
\end{align}
\begin{align}
    \label{eqn:D2-def}
    \bm{D}^2_{jk} = \left\langle V_j^n, \partial_v^2 L^n_k \right\rangle_\Lwinv, 
    \quad \bm{G}_{jk} = \left\langle V_j^n, \partial_v (v L_k^n) \right\rangle_\Lwinv
\end{align}
\textbf{Cost:} $\mathcal{O}(3r(N_v + N_x) + r^2N_v)$, by taking advantage of the rank-3 structure of $\mathcal{N}$.

Then the Forward Euler step for \( K \) can be written concisely as
\begin{align*}
    \frac{K_j^{n+1} - K_j^n}{\Delta t} &= -\bm{y}^1_j - \qm \bm{y}^2_j + \bm{y}^3_j - \sum_l \bm{A}_{jl} \partial_x K^n_l + \sum_k X^n_k \left[ -E^* \bm{D}^1_{jk} + \nu (T^n \bm{D}^2_{jk} + \bm{G}_{jk} - u^n \bm{D}^1_{jk}) \right].
\end{align*}
\textbf{Cost:} $\mathcal{O}(r^2N_x)$.
\begin{remark}
    Additional care must be taken in the implementation of the second-order integrator
    to ensure each of the intermediate vectors and matrices in \eqref{eqn:y1-def} is 
    defined and computed in terms of quantities with the correct time levels on the
    right-hand side.
\end{remark}

\subsubsection*{S step}
Because the $X_i$ basis has been updated in the K step, we must re-project the boundary conditions in \( v \) onto $X_i^{n+1}$.
Define \( \tilde{L}_i = \sum_j S^\prime_{ij} V^n_j \), then the projected boundary conditions are
\begin{align*}
\tilde{L}_i(v_b) = \left\langle X^{n+1}_i, -\mathcal{N}^n(x, v_b) \right\rangle_x.
\end{align*}
To expand \eqref{eqn:S-step}, substitute \eqref{eqn:g-ansatz} into \eqref{eqn:S-step} to obtain
\begin{align}
    \label{eqn:S-step-2}
    \frac{S^{\prime\prime}_{ij} - S^\prime_{ij}}{\Delta t} 
    &= -\left\langle X_i^{n+1} V_j^n, D\left[E^*, \mathcal{N}^n, \sum_{kl} X_k^{n+1} S_{kl}^\prime V_l^n \right] \right\rangle_\xLwinv \\
    &= \left\langle X^{n+1}_i V^n_j, v \Dx \mathcal{N}^n + 
    \qm E^* \Dv \mathcal{N}^n \right\rangle_\xLwinv - \left\langle X^{n+1}_i V^n_j, Q^n(\mathcal{N}^n) \right\rangle_\xLwinv \nonumber \\
&\quad + \sum_{kl} S_{kl}^\prime \left\langle X^{n+1}_i, \Dx X^{n+1}_k \right\rangle_x \left\langle V^n_j, v V^n_l \right\rangle_\Lwinv + 
\qm \sum_{k} \left\langle X^{n+1}_i, E^* X^{n+1}_k \right\rangle_x \left\langle V^n_j, \Dv \tilde{L}_k \right\rangle_\Lwinv \nonumber \\
&\quad - \nu \sum_k \left[ \left\langle X^{n+1}_i, T^n X^{n+1}_k \right\rangle_x \left\langle V^n_j, \Dv^2 \tilde{L}_k \right\rangle_\Lwinv \right. \\
&\qquad + \left\langle X^{n+1}_i, X^{n+1}_k \right\rangle_x \left\langle V^n_j, \Dv (v\tilde{L}_k) \right\rangle_\Lwinv \nonumber \\
&\qquad - \left. \left\langle X^{n+1}_i, u^n X^{n+1}_k \right\rangle_x \left\langle V^n_j, \Dv \tilde{L}_k \right\rangle_\Lwinv \right] \nonumber
\end{align}
We can save computational effort again by precomputing certain matrices which will appear in the L step.
First, compute
\begin{align*}
    \bm{\tilde{D}}^1_{jk} = \left\langle V^n_j, \partial_v \tilde{L}_k \right\rangle_\Lwinv
    \quad \bm{\tilde{D}}^2_{jk} = \left\langle V_j^n, \partial_v^2 \tilde{L}_k \right\rangle_\Lwinv, 
    \quad \bm{\tilde{G}}_{jk} = \left\langle V_j^n, \partial_v (v \tilde{L}_k) \right\rangle_\Lwinv.
\end{align*}
\textbf{Cost:} $\mathcal{O}(r^2 N_v)$.

Then take advantage of the rank-3 structure of $\mathcal{N}$ to compute the inner products in both \( x \) and \( v \):
\begin{align*}
    \bm{Z}^1_{ij} = \left\langle X_i^{n+1} V_j^n, v \Dx \mathcal{N}^n \right\rangle_\xLwinv, 
    \quad \bm{Z}^2_{ij} = \left\langle X_i^{n+1} V_j^n, E^* \Dv \mathcal{N}^n \right\rangle_\xLwinv, 
\end{align*}
\begin{align*}
    \bm{Z}^3_{ij} = \left\langle X_i^{n+1} V_j^n, Q^n(\mathcal{N}^n) \right\rangle_\xLwinv.
\end{align*}
\textbf{Cost:} $\mathcal{O}(3 r^2 (N_v + N_x))$.

Finally, compute the matrices from inner products in \( x \):
\begin{align*}
    \bm{B}_{ik} = \left\langle X_i^{n+1}, \partial_x X^{n+1}_k \right\rangle_x, 
    \quad \bm{F}_{ik} = \left\langle X^{n+1}, E^* X^{n+1}_k \right\rangle_x 
\end{align*}
\begin{align*}
    \bm{N}_{ik} = \left\langle X_i^{n+1}, T^n X^{n+1}_k \right\rangle_x, 
    \quad \delta_{ik} = \left\langle X^{n+1}, X^{n+1}_k \right\rangle_x, 
    \quad \bm{Q}_{ik} = \left\langle X^{n+1}, u^n X^{n+1}_k \right\rangle_x.
\end{align*}
\textbf{Cost:} $\mathcal{O}(r^2 N_x)$.

Now the Forward Euler step for \( S \) can be written concisely as
\begin{align*}
    \frac{S^{\prime\prime}_{ij} - S^\prime_{ij}}{\Delta t} = \bm{Z}^1_{ij} + \qm \bm{Z}^2_{ij} - \bm{Z}^3_{ij} + \sum_{kl} S_{kl}^\prime \bm{B}_{ik} \bm{A}_{jl}
    - \sum_k \left[ - \bm{F}_{ik} \tilde{\bm{D}}^1_{jk} + \nu(\bm{N}_{ik} \tilde{\bm{D}}^2_{jk} + \delta_{ik}\tilde{\bm{G}}_{jk} - \bm{Q}_{ik}\tilde{\bm{D}}^1_{jk}) \right]
\end{align*}
\textbf{Cost:} $\mathcal{O}(r^4)$.

\subsubsection*{L step}
The L step is calculated similarly to the K and S steps.
Because \( \ppp \) involves projecting only in \( v \), it commutes with the \( x \) inner product \( \left\langle \cdot \right\rangle_x \),
which lets us pull the projection \( \ppp \) out of each low-rank projection.
Recall that $\ppp = I - P_\Phi$ where $P_\Phi$ is defined in \eqref{eqn:P_phi_def}
Plug \eqref{eqn:g-ansatz} into \eqref{eqn:L-eqn} to obtain 
\begin{align}
    \label{eqn:L-step-expanded}
\frac{L^{n+1}_i - L_i^n}{\Delta t} 
&= \left\langle X_i^{n+1}, \ppp D\left[E^*, \mathcal{N}^n, \sum_k X^{n+1}_k L_k^n \right] \right\rangle_x \\
&= -\ppp \left\langle X^{n+1}_i, (v \Dx \mathcal{N}^n + \qm E^* \Dv \mathcal{N}^n) \right\rangle_x + \ppp \left\langle X^{n+1}_i, Q^n(\mathcal{N}^n) \right\rangle_x \nonumber \\
&\quad - \ppp \left( \sum_{k} v \left\langle X^{n+1}_i, \Dx X^{n+1}_k \right\rangle_x L^n_k + \sum_k \left\langle X^{n+1}_i, E^* X^{n+1}_k \right\rangle_x \Dv L^n_k \right) \\
&\quad + \ppp \nu \sum_k \left[ \left\langle X^{n+1}_i, T^n X^{n+1}_k \right\rangle_x \Dv^2 L^n_k + \left\langle X^{n+1}_i, X^{n+1}_k \right\rangle_x \Dv (vL^n_k) - \left\langle X^{n+1}_i, u^n X^{n+1}_k \right\rangle_x \Dv L^n_k \right]. \nonumber
\end{align}
Compute vectors appearing in \eqref{eqn:L-step-expanded}:
\begin{align*}
\bm{z}^1_i(v) = \left\langle X^{n+1}_i, v \Dx \mathcal{N}^n \right\rangle_x, 
\quad \bm{z}^2_i(v) = \left\langle X^{n+1}_i, E^* \Dv \mathcal{N}^n \right\rangle_x, 
\quad \bm{z}^3_i(v) = \left\langle X_i, Q(\mathcal{N}^n; U) \right\rangle_x.
\end{align*}
\textbf{Cost:} $\mathcal{O}(3r(N_v + N_x)$.

The Forward Euler step for \( L \) can be written concisely as
\begin{align}
    \label{eqn:L-step-3}
\frac{L^{n+1}_i - L^n_i}{\Delta t} &= \ppp (-\bm{z}^1_i(v) - \qm \bm{z}^2_i(v) + \bm{z}^3_i(v)) \\
& \quad + \ppp \left[ \sum_k (-v L^n_k \bm{B}_{ik} - \qm \bm{F}_{ik} \partial_v L^n_k + \nu (\bm{N}_{ik} \partial_v^2 L^n_k + \delta_{ik} \partial_v (v L^n_k) - \bm{Q}_{ik} \partial_v L^n_k )) \right].
\end{align}
\textbf{Cost:} $\mathcal{O}(r^2 N_v)$.

The total computational cost of all three substeps scales as $\mathcal{O}(r^2(N_x + N_v) + r^4)$.
The asymptotic advantages do not begin to tell in one dimension, but in 2D2V and 3D3V kinetic applications, we can easily have $N_x, N_v \gg r^2$,
in which case the efficiency gains from dynamical low-rank approximation are quite substantial.

\section{Spatial discretization}
\label{sec:spatial-discretization}
In this section we discuss the discretization of the equations of Section \ref{sec:time-discretization} in \( x \).
We will continue to leave the \( v \)-discretization unspecified for now, continuing our formulation in terms of continuous inner products in \( v \).
In $x$, we use a simple finite difference scheme on an equispaced grid, with grid separation denoted \( \Delta x \),
and points \( x_i \).
Functions are approximated directly on the grid, so that \( u(x_i) \approx u_i \).
For discretization of the inner product \( \left\langle \cdot \right\rangle_x \) over physical space,
we use the Trapezoidal rule.
On periodic domains, the Trapezoidal rule has spectral convergence, while for nonperiodic domains,
its second-order convergence is sufficient and matches the order of accuracy of our
finite difference discretization of the hyperbolic terms, described below.
The Poisson equation is solved using a standard centered finite difference stencil,
also of second order accuracy.
Finally, for non-hyperbolic first-order derivatives appearing inside of an inner product in \( x \), we use
a second-order centered finite difference approximation to \( \partial_x \).
Discussion of our hyperbolic finite difference discretization is below.

\subsection{Finite difference discretization of hyperbolic terms}
Because of the complexity of our time discretization, additional care is required to ensure that the overall hyperbolic structure of the kinetic equation is preserved in our discretization of \( x \) and \( v \).
In particular, four of the equations described in Section \ref{sec:time-discretization} have an advective form, involving the \( x \)-derivative of one or more unknowns.
Ignoring the other (source) terms in those equations, they are:
\begin{align}
    \label{eqn:first-3-rhs}
\partial_t \begin{pmatrix}
f_0 \\ f_1 \\ f_2
\end{pmatrix} + \begin{pmatrix}
b_0 & a_0 & 0 \\
a_0 & b_1 & a_1 \\
0 & a_1 & b_2
\end{pmatrix} \partial_x \begin{pmatrix}
f_0 \\ f_1 \\ f_2
\end{pmatrix} + \partial_x \begin{pmatrix}
0 \\ 0 \\ a_2 f_3
\end{pmatrix} = RHS,
\end{align}
and
\begin{align}
    \label{eqn:k-rhs}
\partial_t K_j + \left\langle V_j, v \partial_x \mathcal{N} \right\rangle_\Lwinv + \sum_l \bm{A}_{jl} \partial_x K_l = RHS,
\end{align}
where $\bm{A}_{jl}$ is defined according to \eqref{eqn:Ajl-def}.

\begin{prop}
    Equations \eqref{eqn:first-3-rhs} and \eqref{eqn:k-rhs} together form a globally hyperbolic system of
    partial differential equations.  
\end{prop}
\begin{proof}
Note that, by \eqref{eqn:p-recurrence}, 
\begin{align*}
\left\langle V_j, v \partial_x \mathcal{N} \right\rangle_\Lwinv = \left\langle V_j, a_2 \partial_x f_2(x, t) w(v) p_3(v) \right\rangle_\Lwinv = \left\langle V_j, w(v) p_3(v) \right\rangle_\Lwinv a_2 \partial_x f_2(x, t),
\end{align*}
where we have used Proposition \ref{prop:Vj_projection} to eliminate terms in the span of $w(v) \mathbf{p}(v)$ inside the inner product.
If we abbreviate the inner product by \( \left\langle V_j, w(v) p_3(v) \right\rangle_\Lwinv = \bm{\hat{y}}_j \),
we can write the combined system in quasilinear form like so:
\begin{align}
    \label{eqn:quasilinear}
\partial_t \begin{pmatrix}
f_0 \\ f_1 \\ f_2 \\ K_1 \\ K_2 \\ \vdots \\ K_r
\end{pmatrix} + 
\underbrace{\begin{pmatrix}
b_0 & a_0 \\
a_0 & b_1 & a_1 \\
& a_1 & b_2 & a_2 q_1 & a_2 q_2 & \dots & a_2 q_r\\
& & a_2 \hat{\bm{y}}_1 & \multicolumn{4}{c}{\multirow{4}{*}{\Large$\bm{A}_{jl}$\normalsize}} \\
& & a_2 \hat{\bm{y}}_2 \\
& & \vdots \\
& & a_2 \hat{\bm{y}}_r
\end{pmatrix}}_{\bm{A}_0} \partial_x \begin{pmatrix}
f_0 \\ f_1 \\ f_2 \\ K_1 \\ K_2 \\ \vdots \\ K_r
\end{pmatrix} = RHS.
\end{align}
Here we have rewritten \( f_3(x) \) as a linear combination of the \( K_j(x) \), with coefficients \( q_j \).
Per \eqref{eqn:f-3}, we have \( q_j = \left\langle p_3(v), V_j \right\rangle_v \).
But then
\begin{align*}
\hat{\bm{y}}_j = \left\langle V_j, w(v) p_3(v) \right\rangle_\Lwinv = \left\langle V_j, p_3(v) \right\rangle_v = q_j.
\end{align*}
Additionally, it is clear from its definition that \( \bm{A}_{jl} \) is symmetric.
Therefore, the matrix \( \bm{A}_0 \) is symmetric, and thus is strictly hyperbolic with real eigenvalues.
\end{proof}
The hyperbolicity of the overall system gives us confidence in applying standard discretization techniques for
hyperbolic conservation laws.
As a spatial discretization, we opt for a Shu-Osher conservative finite difference discretization based on upwind flux splitting.
The hyperbolic system is always upwinded ``all at once'', despite the fact that we solve the first three and the trailing $r$ equations separately.

To be precise, fix $V_j$ and let $R \Lambda R^{-1} = \bm{A}_0$ be the eigendecomposition of $\bm{A}_0(V_j)$.
Denote by $\bm{w}$ the length-$r+3$ vector $\bm{w} = [f_0, f_1, f_2, K_1, \dots, K_r]^T$.
Let $P_1$ and $P_2$ be the projection matrices consisting of the first three and trailing $r$ rows of an $(r + 3) \times (r+3)$ identity matrix; they select the rows of $\bm{w}$ corresponding to $f_0, f_1, f_2$ and $K_j$ respectively:
\begin{align}
    P_1 \bm{w} = [f_0, f_1, f_2]^T, \quad P_2 \bm{w} = [K_1, K_2, \dots, K_r]^T.
\end{align}
The flux terms appearing in \eqref{eqn:first-3-rhs} can be written as
\begin{align}
\begin{pmatrix}
b_0 & a_0 & 0 \\
a_0 & b_1 & a_1 \\
0 & a_1 & b_2
\end{pmatrix} \partial_x \begin{pmatrix}
f_0 \\ f_1 \\ f_2
\end{pmatrix} + \partial_x \begin{pmatrix}
0 \\ 0 \\ a_2 f_3
\end{pmatrix}
=
P_1 \partial_x \bm{A}_0 \bm{w},
\end{align}
and the flux terms in \eqref{eqn:k-rhs} as
\begin{align}
    \left\langle V_j, v \partial_x \mathcal{N} \right\rangle_\Lwinv + \sum_l \bm{A}_{jl} \partial_x K_l = P_2 \partial_x \bm{A}_0 \bm{w}.
\end{align}

The combined flux of $f_0, f_1, f_2,$ and $K_j$ is $F(\bm{w}) = \bm{A}_0 \bm{w}$.
Notionally, we approximate $\partial_x F(\bm{w})$ by a conservative flux difference:
\begin{align*}
\partial_x F(\bm{w}^n) = \frac{1}{\Delta x} (\hat{F}_{i+1/2} - \hat{F}_{i-1/2}),
\end{align*}
where $\hat{F}_{i+1/2}$ is the numerical flux between cells $i$ and $i+1$.
The numerical flux is split and upwinded according to the eigendecomposition of $\bm{A}_0$.
For example, a first-order upwind finite difference scheme would use
\begin{align}
    \label{eqn:upwind_fd_split}
    \hat{F}_{i+1/2} = \hat{F}^-_{i+1/2} + \hat{F}^+_{i+1/2} = R \Lambda^- R^{-1} \bm{w}_{i+1} + R \Lambda^+ R^{-1} \bm{w}_i.
\end{align}
In this work we use higher-order spatial reconstructions of the split flux, following
the conservative finite difference framework of \cite{shuEssentiallyNonoscillatoryWeighted1998}.
Except where noted, we use a MUSCL reconstruction with the Monotonized-Central limiter which is second-order in space.
We then approximate the hyperbolic terms in \eqref{eqn:first-3-rhs} and \eqref{eqn:k-rhs} as
\begin{align}
    P_s \partial_x F(\bm{w}) =
    P_s \frac{1}{\Delta x} \left( \hat{F}_{i+1/2} - \hat{F}_{i-1/2} \right)
\end{align}
for $P_s \in \{ P_1, P_2 \}$.
Importantly, we apply the projection $P_s$ \emph{after} the flux splitting and upwind reconstruction.
We base the upwinding procedure on the whole size $r+3$ eigenvector decomposition of $\bm{w}$, rather than
upwinding systems of size $3$ and size $r$ separately.

\section{Two velocity space discretizations}
\label{sec:velocity-discretization}

One of the benefits of our method is that it is generic over the choice of
velocity space discretization.
To illustrate this, we present a pair of velocity space discretizations based on orthogonal 
polynomial expansions, for which the polynomials $p_0, p_1, p_2$ appear as the first three
polynomials in one of the classical orthogonal polynomial families.

The first discretization we will discuss is a global Hermite spectral method, for which $p_0, p_1, p_2$ naturally
enter as the first three Hermite polynomials.
This method lets us discretize an unbounded velocity domain without arbitrary truncation.
The second discretization uses a truncated velocity space, with $p_0, p_1, p_2$ chosen as the first three scaled Legendre polynomials.
The resulting equations for the non-conserved part $g$ are solved with a standard upwind finite difference method.
Our aim with this pair of discretizations is to demonstrate the flexibility of the underlying macro-micro decomposition, which may be combined with whatever velocity space discretization is most convenient for the problem at hand.

\subsection{Asymmetrically-weighted Hermite spectral method}
This section derives a global spectral expansion for the low-rank component \( g \) in terms of the asymmetrically weighted, normalized Hermite polynomials.
Expansions in terms of Hermite polynomials enjoy a long history in numerical methods for kinetic equations.
An important distinction is that between ``centered'' Hermite expansions, which expand in polynomials which are orthogonal \emph{with respect to the local Maxwellian}, and ``uncentered'' expansions.
The latter use a fixed Gaussian distribution as the weight function for the polynomial family, and pay for their greatly increased simplicity with less rapid convergence.
Perhaps the most famous example of a centered Hermite expansion is Grad's moment method \cite{gradKineticTheoryRarefied1949}.
Our method is an uncentered method, of which many examples have been developed in recent years.
The use of an uncentered global Hermite expansion in velocity space may be coupled with a choice of physical space discretization; we highlight examples using Discontinuous Galerkin \cite{koshkarovMultidimensionalHermitediscontinuousGalerkin2021, filbetConservativeDiscontinuousGalerkin2022}
and Fourier spectral \cite{delzannoMultidimensionalFullyimplicitSpectral2015, vencelsSpectralPlasmaSolverSpectralCode2016} methods in \( x \).
As described in Section \ref{sec:spatial-discretization}, in this work we use a flux-limited high-resolution conservative finite difference scheme in $x$.

In terms of our notation, the Hermite polynomials are the orthogonal polynomial family defined on \( \Omega_v = \mathbb{R} \) with weight function
\begin{align*}
w(v) = \frac{1}{v_0 \sqrt{2\pi}} e^{-\frac{v^2}{2v_0^2}},
\end{align*}
where \( v_0 \) is a reference velocity which sets the width of the basis.
The Hermite polynomials satisfy the following orthogonality relation:
\begin{align*}
\int_{\Omega_v} w(v) He_n\left( \frac{v}{v_0} \right)He_m\left( \frac{v}{v_0} \right)\,\mathrm{d}v = \delta_{nm}.
\end{align*}
Their three-term recurrence relation has coefficients \( a_n = v_0 \sqrt{n+1} \) and \( b_n = 0 \) \cite{DLMF}.
The first several Hermite polynomials are
\begin{align*}
&p_0(v) = He_0 \left( \frac{v}{v_0} \right) = 1 
&&p_2(v) = He_2 \left( \frac{v}{v_0} \right) = \frac{(v/v_0)^2 - 1}{\sqrt{2}} \\
&p_1(v) = He_1 \left( \frac{v}{v_0} \right) = v/v_0
&&p_3(v) = He_3 \left( \frac{v}{v_0} \right) = \frac{(v/v_0)^3 - v/v_0}{\sqrt{6}}.
\end{align*}
From these definitions, it is simple to verify the identities
\begin{align*}
    v = v_0 p_1(v), \quad \frac{v^2}{2} = \frac{v_0^2}{2} (\sqrt{2} p_2(v) + 1),
\end{align*}
or in terms of the matrix $C$,
\begin{align*}
C = \begin{pmatrix}
    1 & 0 & 0 \\
    0 & v_0 & 0 \\
    \frac{v_0^2}{2} & 0 & \frac{v_0^2}{\sqrt{2}}
\end{pmatrix}.
\end{align*}
The derivative coefficients are
\begin{align*}
    d_{10} = \frac{1}{v_0}, \quad d_{20} = 0, \quad d_{21} = \frac{\sqrt{2}}{v_0}.
\end{align*}

We search for a solution \( f = \mathcal{N} + g \) where \( g \) is expanded in terms of the first $M+1$ Hermite polynomials,
\begin{align*}
g(x, v, t) = \sum_{n=0}^M w(v) He_n\left( \frac{v}{v_0} \right) g_n(x, t) = \bm{H}_M^T \bm{g}.
\end{align*}
We have expressed the sum in our preferred notation, which considers the sequence of asymmetrically weighted Hermite
polynomials as a vector which may be combined via a dot product with the vector of coefficients, \( \bm{g}(x, t) \).

The weighted inner product \( \left\langle \cdot \right\rangle_\Lwinv \) is discretized as the discrete
dot product of two coefficient vectors, as demonstrated by the following sequence of identities:
\begin{align*}
\int_{\Omega_v} w^{-1}(v) g(v) h(v) \,\mathrm{d}v &= \int_{\Omega_v} w^{-1}(v) \left(\sum_{n=0}^M w(v) He_n\left( \frac{v}{v_0} \right) g_n \right) \left( \sum_{m=0}^M w(v) He_m\left( \frac{v}{v_0} \right) h_m \right) \,\mathrm{d}v \\
&= \sum_{n=0}^M \sum_{m=0}^M g_n h_m \int_{\Omega_v} w(v) He_n\left( \frac{v}{v_0} \right) He_m\left( \frac{v}{v_0} \right)\,\mathrm{d} v \\
&= \sum_{n=0}^M \sum_{m=0}^M g_n h_m \delta_{nm} \\
&= \bm{g}^T \bm{h}.
\end{align*}
The required differentiation operators in the normalized Hermite basis are discretized by the following matrices,
which may be derived from the recurrence relations for the weighted Hermite polynomials \cite{DLMF}:
\begin{align*}
\partial_v (\bm{H}_M^T \bm{g}) = \frac{1}{v_0} \bm{H}_M^T \underbrace{\begin{pmatrix}
0 \\
-1 \\
& -\sqrt{2} \\
& & -\sqrt{3} \\
&  & & \ddots
\end{pmatrix}}_{\mathcal{D}} \bm{g}, \quad
\partial_v (v (\bm{H}_M^T \bm{g})) = \bm{H}_M^T \underbrace{\begin{pmatrix}
0 \\
0 & -1 \\
-\sqrt{2} & 0 & -2 \\
& -\sqrt{6} & 0 & -3 \\
& & \ddots & & \ddots
\end{pmatrix}}_{\mathcal{D}\Xi} \bm{g},
\end{align*}
\begin{align}
    \label{eqn:defn-Xi}
\partial_v^2 (\bm{H}_M^T \bm{g}) = \frac{1}{v_0^2} \bm{H}_M^T \underbrace{\begin{pmatrix}
0 \\
0 \\
\sqrt{2} \\
& \sqrt{6} \\
& & & \ddots
\end{pmatrix}}_{\mathcal{D}^2} \bm{g}, \quad
v (\bm{H}_M^T \bm{g}) = v_0 \bm{H}_M^T \underbrace{\begin{pmatrix}
0 & 1 \\
1 & 0 & \sqrt{2} \\
& \sqrt{2} & 0 & \sqrt{3} \\
& & \ddots & & & \ddots
\end{pmatrix}}_{\Xi} \bm{g}
\end{align}
Finally, the projection \( P_\Phi \) can be discretized as the operation which
discards all but the first three Hermite coefficients, and conversely \( \ppp \) as the
operation which sets the first three Hermite coefficients to zero:
\begin{align*}
P_\Phi f(v) = \bm{H}_M^T \begin{pmatrix}
f_0 \\ f_1 \\ f_2 \\ 0 \\ 0 \\ \vdots
\end{pmatrix}, \quad \ppp f(v) = \bm{H}_M^T \begin{pmatrix}
0 \\ 0 \\ 0 \\ f_3 \\ f_4 \\ \vdots
\end{pmatrix}
\end{align*}
The discrete operators defined above completely specify a velocity space discretization of
our scheme.

\subsubsection{Spectral filtering of Hermite modes}
In order to avoid numerical instability due to the Gibbs phenomenon, we apply a filter to the vector of Hermite
modes after each timestep.
Following \cite{filbetConservativeDiscontinuousGalerkin2022}, we employ the filter known as Hou-Li's filter \cite{houComputingNearlySingular2007} which prescribes multiplying the $m^{\text{th}}$ Hermite mode by a
scaling factor $\sigma \left( \frac{m}{M+1} \right)$, where
\begin{align}
    \sigma(s) = \begin{cases}
        1, & 0 \leq s \leq 2/3, \\
        e^{-\beta s^\beta} & s > 2/3,
    \end{cases}
\end{align}
with $\beta = 36$ designed to eliminate the final mode to within machine precision.
This filter is applied after each completed L step, i.e. step 5 in the first-order integrator and steps 5-6 in the second-order integrator.

\subsection{Truncated domain finite difference method with Legendre weight}
While global spectral methods for velocity space such as the Hermite method exhibit high accuracy
and are easy to implement, they are not the only option available, nor the best in all circumstances.
Grid-based methods, in particular Discontinuous Galerkin methods, are often preferred for
their better resolution of fine filamentation structures in velocity space 
\cite{hoPhysicsBasedAdaptivePlasmaModel2018, hakimAliasFreeMatrixFreeQuadratureFree2020}.
In particular, unlike Hermite spectral methods, grid-based methods do not suffer from 
degraded resolution when the drift velocity or temperature of the local solution differs too much from 
the ``reference'' velocity and temperature around which the polynomial basis is expanded.

To illustrate the flexibility of our scheme to accomodate a variety of velocity discretizations
including grid-based methods, 
in this section we describe a finite difference discretization of a truncated velocity space
with the macro-micro decomposition based on the Legendre polynomials $P_i(v)$.

Our truncated velocity domain is \( \Omega_v = [-v_{max}, v_{max}] \), with 
the constant weight function
\begin{align*}
w(v) = \frac{1}{v_{max}}.
\end{align*}
The Legendre polynomials scaled to \( \Omega_v \) satisfy the orthogonality relation
\begin{align*}
\int_{\Omega_v} w(v) P_n\left( \frac{v}{v_{max}} \right) P_m \left( \frac{v}{v_{max}} \right) \,\mathrm{d} v = \delta_{nm}.
\end{align*}
The recurrence relation for the orthonormal Legendre polynomials has coefficients
\begin{align*}
a_n = \frac{n+1}{\sqrt{(2n+1)(2n+3)}} v_{max}, \quad b_n = 0,
\end{align*}
and the first several examples are
\begin{align*}
&p_0(v) = P_0\left( \frac{v}{v_{max}} \right) = \sqrt{\frac{1}{2}} 
&&p_2(v) = P_2\left( \frac{v}{v_{max}} \right) = \sqrt{\frac{5}{8}} \left( 3\left(\frac{v}{v_{max}}\right)^2 - 1 \right) \\
&p_1(v) = P_1 \left( \frac{v}{v_{max}} \right) = \sqrt{\frac{3}{2}} \frac{v}{v_{max}}
&&p_3(v) = P_3\left( \frac{v}{v_{max}} \right) = \sqrt{\frac{7}{8}} \left( 5\left(\frac{v}{v_{max}}\right)^3 - 3\frac{v}{v_{max}} \right).
\end{align*}
From the first three Legendre polynomials it is easy to verify the identities
\begin{align*}
    v = \sqrt{\frac{2}{3}} v_{max} p_1(v), \quad \frac{v^2}{2} = \frac{v_{max}^2}{6} \left(\sqrt{\frac{8}{5}} p_2(v) + 1\right),
\end{align*}
or in terms of the matrix $C$,
\begin{align*}
C = \begin{pmatrix}
    \sqrt{2} & 0 & 0 \\
    0 & \sqrt{\frac{2}{3}} v_{max} & 0 \\
    \frac{v_{max}^2}{6} & 0 & \sqrt{\frac{8}{5}}\frac{v_{max}^2}{6}
\end{pmatrix}.
\end{align*}
The derivative coefficients $d_{10}, d_{20}, d_{21}$ are as follows:
\begin{align*}
    d_{10} = \frac{\sqrt{3}}{v_{max}}, \quad d_{20} = 0, \quad d_{21} = \frac{\sqrt{15}}{v_{max}}.
\end{align*}
We discretize velocity space using a finite difference discretization.
The hyperbolic term in the L step, equation \eqref{eqn:L-step-3}, is discretized using a piecewise-linear MUSCL reconstruction with
the monotonized-central (MC) slope limiter and a Lax-Friedrichs numerical flux.
Second-order derivatives in $v$ are discretized using a second-order centered finite difference operator.
Inner products in $v$ are computed using the midpoint rule.
Boundary conditions on $g$ are implemented via extrapolation of the Dirichlet boundary condition
$g(x, v_b, t) = -\mathcal{N}(x, v, t)$ into a ghost cell layer that is two cells wide.

\section{Numerical results}
\label{sec:numerical_results}

In this section we provide numerical results from standard benchmark problems in computational plasma physics.
All benchmarks are implemented in both the global Hermite and finite difference velocity discretizations.
For the Hermite discretization we use $v_0 = 1.0$ corresponding to a weight function $w(v) = \frac{1}{\sqrt{2\pi}}e^{-\frac{v^2}{2}}$.
For the finite difference discretization, we use $v_{max} = 8.0$.
Except where otherwise indicated, we use the second-order time integrator and neglect collisions, $\nu = 0.0$.

\subsection{Verification of second-order temporal accuracy}
To verify the claimed second-order accuracy of the splitting scheme described in Section \ref{sec:second-order-integrator},
we perform a convergence study using the weak Landau damping numerical test of Section \ref{sec:weak_landau_damping}.
This is solved using a fifth-order WENO finite difference discretization \cite{shuHighOrderWeighted2009} in space with $N_x=128$ grid points,
and the Hermite spectral discretization in velocity with $M=256$.
The first-order scheme is run with $\Delta t$ ranging from $\num{4e-3}$ to $\num{1.25e-4}$,
while the second-order scheme is run with $\Delta t$ from $\num{8e-3}$ to $\num{5e-4}$.
Convergence is observed by comparing the solution $f(\Delta t)$ with the refined solution
$f(\Delta t / 2)$ at time $t = 5.0$, and taking the $L^2$ norm of the difference.
The results are shown in Figure \ref{fig:convergence_study}.
We observe excellent agreement between the theoretical and empirical rates of convergence for both integrators.
\begin{figure}
    \centering
    \includegraphics[width=0.75\textwidth]{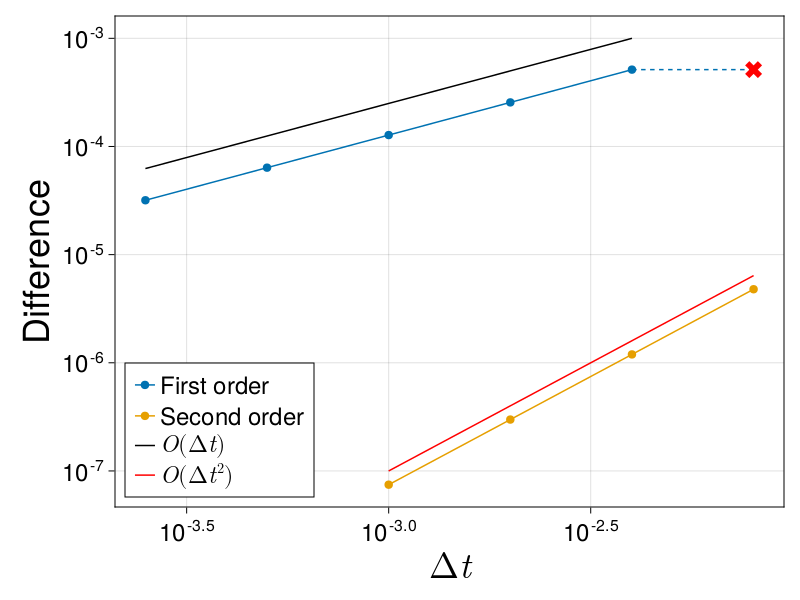}
    \caption{Convergence plot demonstrating first and second-order accuracy in time of the respective time integrators. The first-order integrator
        is unstable for $\Delta t = \num{8e-3}$.
    \label{fig:convergence_study}}
\end{figure}

\subsection{Weak Landau damping}
\label{sec:weak_landau_damping}
As a physics test, we reproduce the ubiquitous weak Landau damping benchmark problem with wavenumber \( k = 0.5 \)
using both the Hermite spectral discretization and the Legendre-weighted finite difference discretization
in velocity.

The initial condition is
\begin{align*}
f_0(x, v) = (1 + \delta \cos(k x)) e^{-v^2/2}, \quad x \in (0, 2\pi / k).
\end{align*}
The perturbation size is set to \( \delta = \num{1e-3} \).
The domain is discretized with \( N_x = 128 \) points in the \( x \) direction and with either \( M=256 \) Hermite modes or $N_v = 256$ velocity grid points.
The rank is set to \( r = 6 \).
We run the simulation with the second-order integrator to time \( t=40 \), using timesteps of \( \Delta t = \num{2e-3} \).
The result is shown in Figure \ref{fig:weak-LD}.
We measure a damping rate of \( \gamma = -0.1525 \) for the Hermite spectral discretization and $\gamma = -0.1523$ for the finite difference discretization of velocity space, demonstrating good agreement with the linear theory prediction of \( \gamma = -0.153 \).

\begin{figure}
    \begin{subfigure}[h]{\textwidth}
        \centering
        \includegraphics*[width=\textwidth]{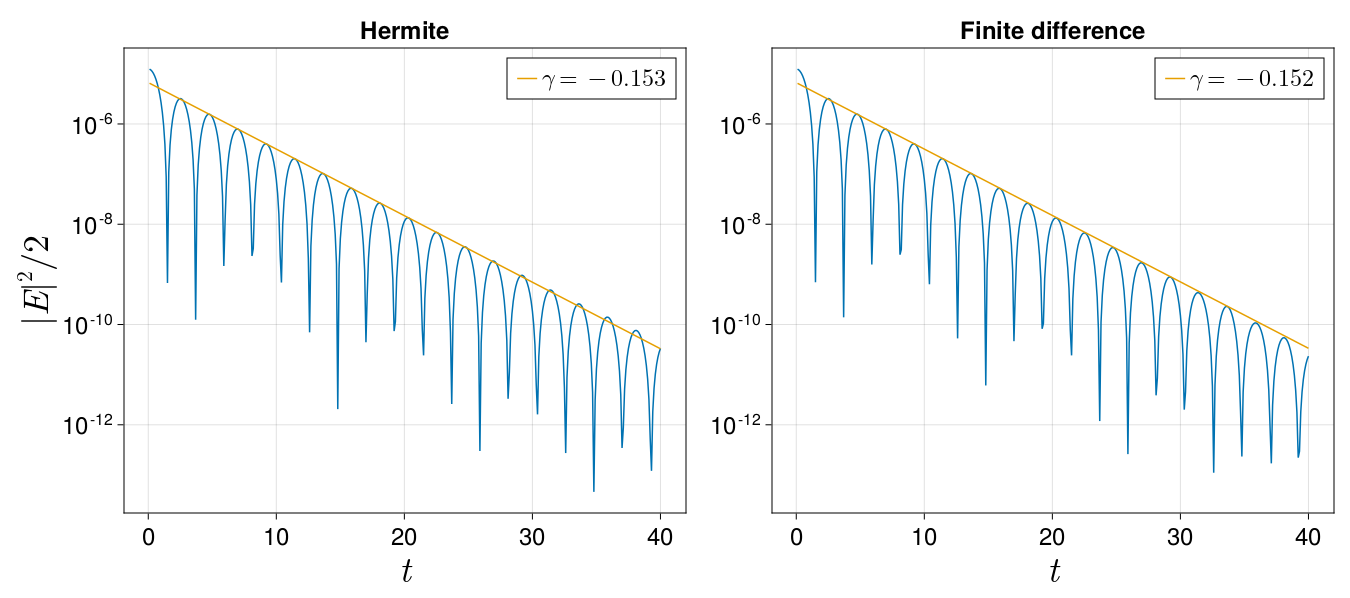}
        \caption{
            Electric energy and best linear fit.
        }
    \end{subfigure}
    \\
    \begin{subfigure}[h]{\textwidth}
        \centering
        \includegraphics*[width=\textwidth]{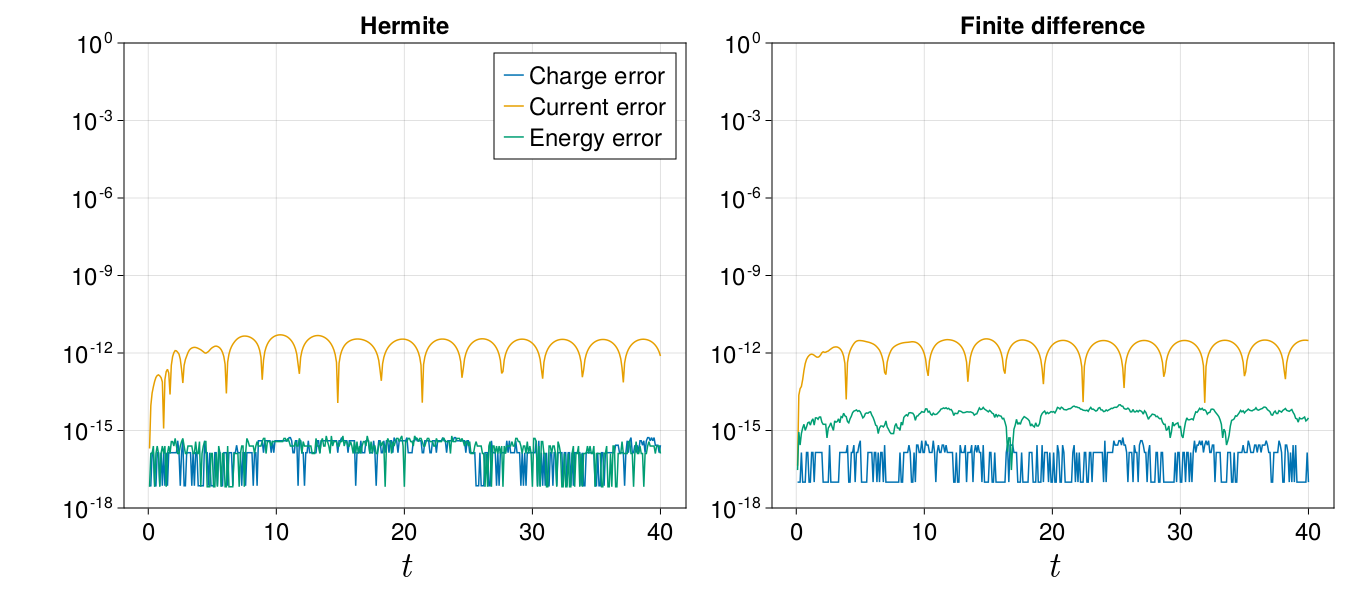}
        \caption{
            Charge, current and energy conservation error.
        }
    \end{subfigure}
    \caption{
        Weak Landau damping example, demonstrating exact conservation of charge and energy with the second-order
        time integrator.
            \label{fig:weak-LD}
    }
\end{figure}

\subsubsection*{Collisional Landau damping}
We validate the correctness of our solver including collisionality by comparing the Landau damping phenomenon
at a variety of collision frequencies $\nu$.
We solve the same weak Landau damping problem as above, but with the collision frequency \( \nu \) set to \( 0.0, 0.25, \) and \( 1.0 \).
The Hermite spectral solver is run with $M = 256$ and $\Delta t = \num{2e-3}$, the same as the collisionless example.
In contrast, the finite difference velocity discretization becomes more stiff as the diffusive collision
term grows larger, so for that discretization we reduce both the velocity grid spacing and timestep to $N_v = 128$
grid points and $\Delta t = \num{5e4}$.
The results are shown in Figure \ref{fig:collisional-LD}.
\begin{figure}
    \begin{subfigure}[h]{\textwidth}
        \centering
        \includegraphics[width=\textwidth]{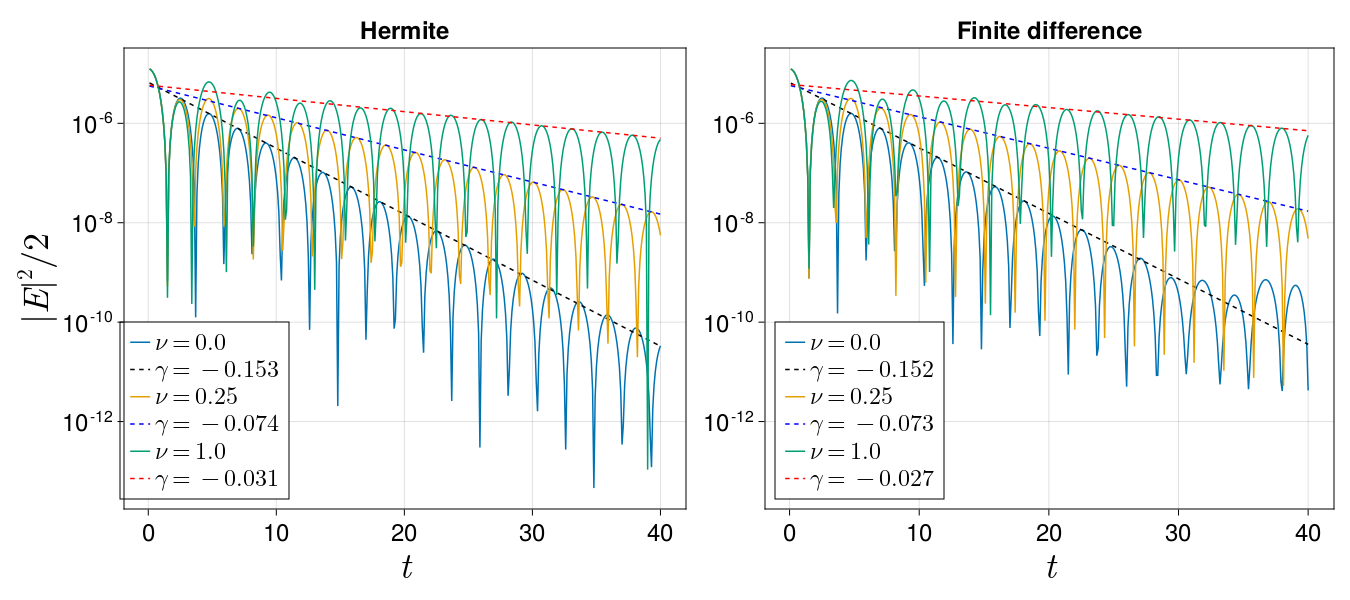}
        \caption{Measured damping rates as a function of collision frequency $\nu$. 
            Increasing collisionality brings the solution nearer to the fluid limit, which experiences no damping.
        Our results agree well with published results in \cite{hakimConservativeDiscontinuousGalerkin2020}.}
     \end{subfigure}
    \begin{subfigure}[h]{\textwidth}
        \centering
        \includegraphics[width=\textwidth]{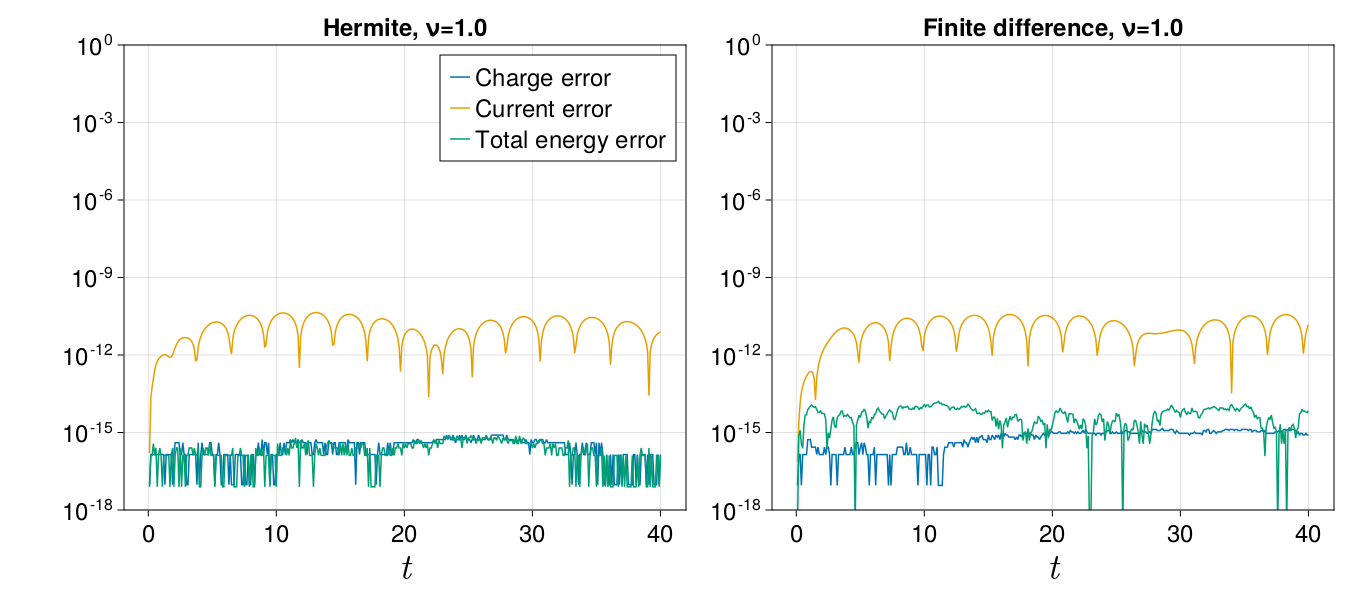}
        \caption{Conservation of charge and energy by the second-order integrator, tested in the $\nu=1.0$ case.}
     \end{subfigure}
     \caption{Collisional Landau damping example. \label{fig:collisional-LD}}
\end{figure}

\subsection{Strong Landau damping}
In this example we use the same wavenumber and domain as the weak Landau damping problem, \( k = 0.5 \), but set \( \delta = 0.5 \) to
explore the strong (nonlinear) Landau damping regime.
Again we present results from both the Hermite spectral and finite difference discretizations in velocity space.
The simulation is run with \( r = 16 \), on a grid with \( N_x=128 \) grid points in \( x \), and either \( M = 256 \) Hermite modes or $N_v=256$ velocity grid points.
The timestep is set to \( \Delta t = \num{4e-3} \), and the initial condition is evolved using the 
second-order energy-conserving integrator to \( t=50.0 \).
The results are shown in Figure \ref{fig:strong_landau_damping}, including the phase space density at $t = 25.0$.
Conservation properties of both first and second order integrators on this strong Landau damping problem are shown in Figure \ref{fig:sld_conservation}.

\begin{figure}
    \begin{subfigure}{\textwidth}
        \centering
        \includegraphics[width=\textwidth]{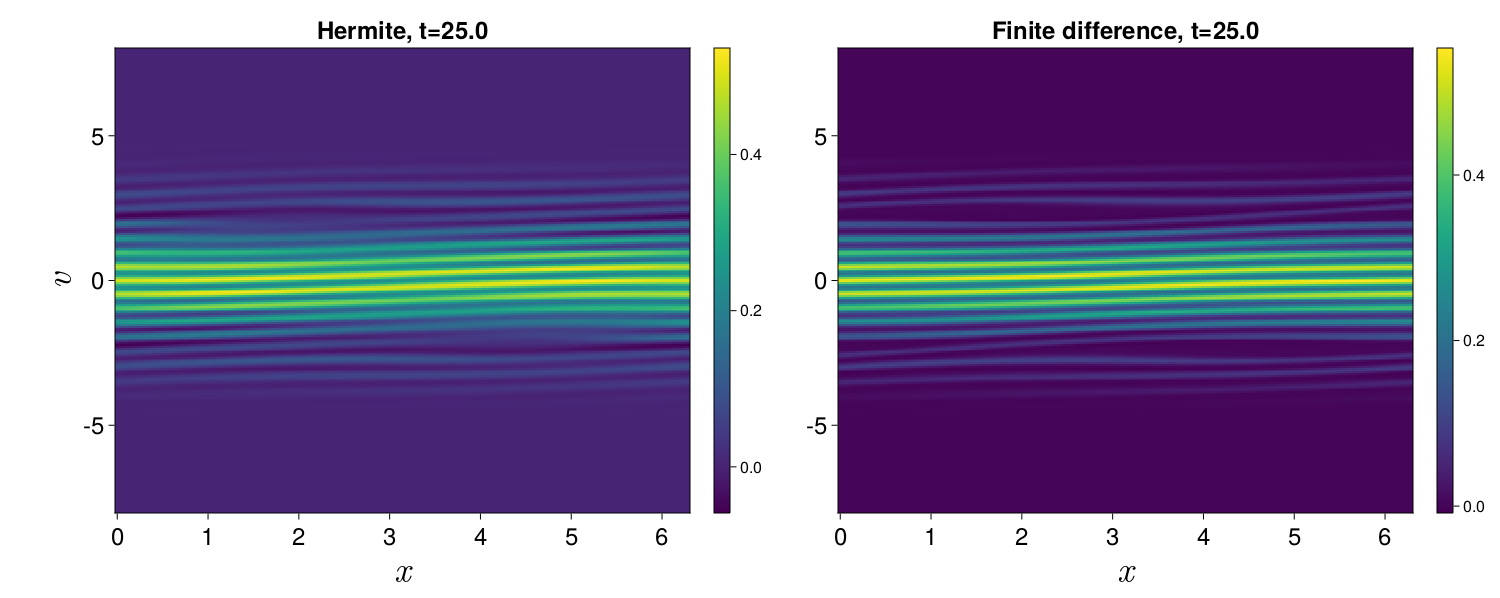}
        \caption{Phase space distribution $f(x, v, t=25.0)$ showing strong filamentation.
            The finite difference discretization is both less diffuse and shows less severe positivity violations than
            the Hermite discretization, which demonstrates the importance of compatibility with a diversity of
            velocity space discretizations.
        }
    \end{subfigure}
    \begin{subfigure}{\textwidth}
        \centering
        \includegraphics[width=\textwidth]{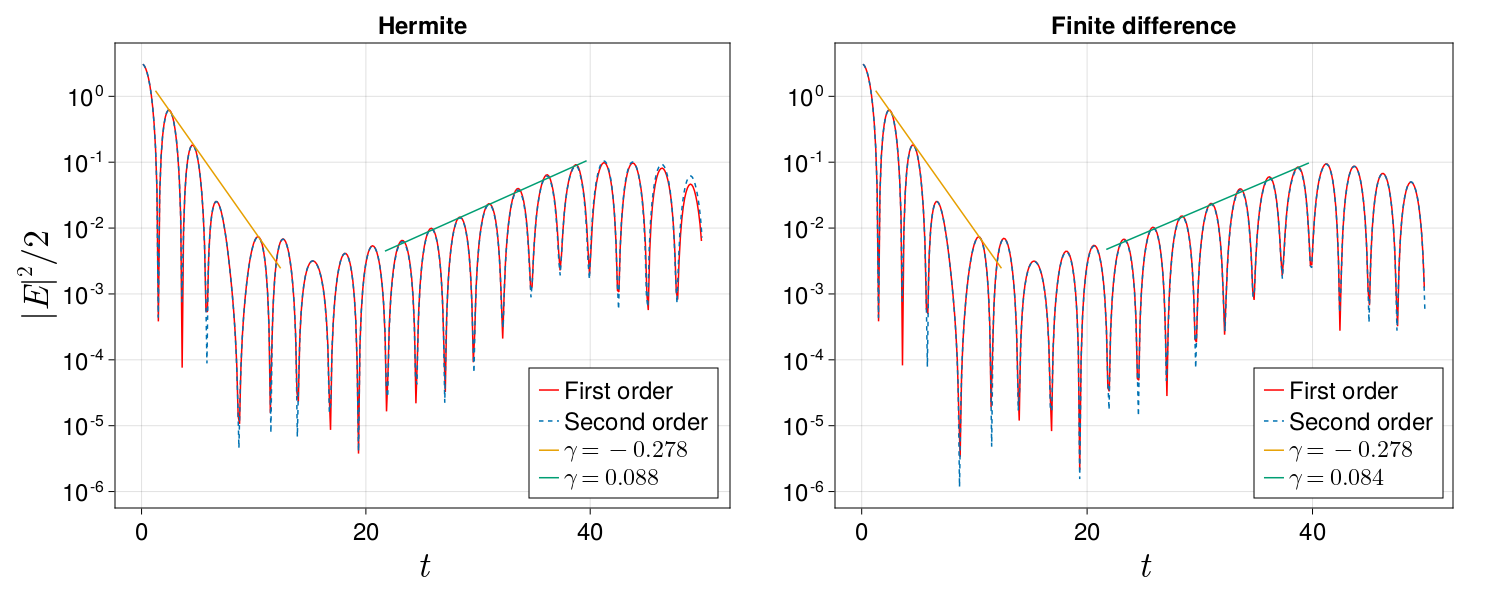}
        \caption{Electric energy traces. Our implementation matches both the initial nonlinear 
            damping and subsequent growth rates, compared to published results such as \cite{hoPhysicsBasedAdaptivePlasmaModel2018}.}
    \end{subfigure}
    \caption{Strong Landau damping example. \label{fig:strong_landau_damping}}
\end{figure}

\begin{figure}
    \begin{subfigure}{\textwidth}
        \centering
        \includegraphics[width=\textwidth]{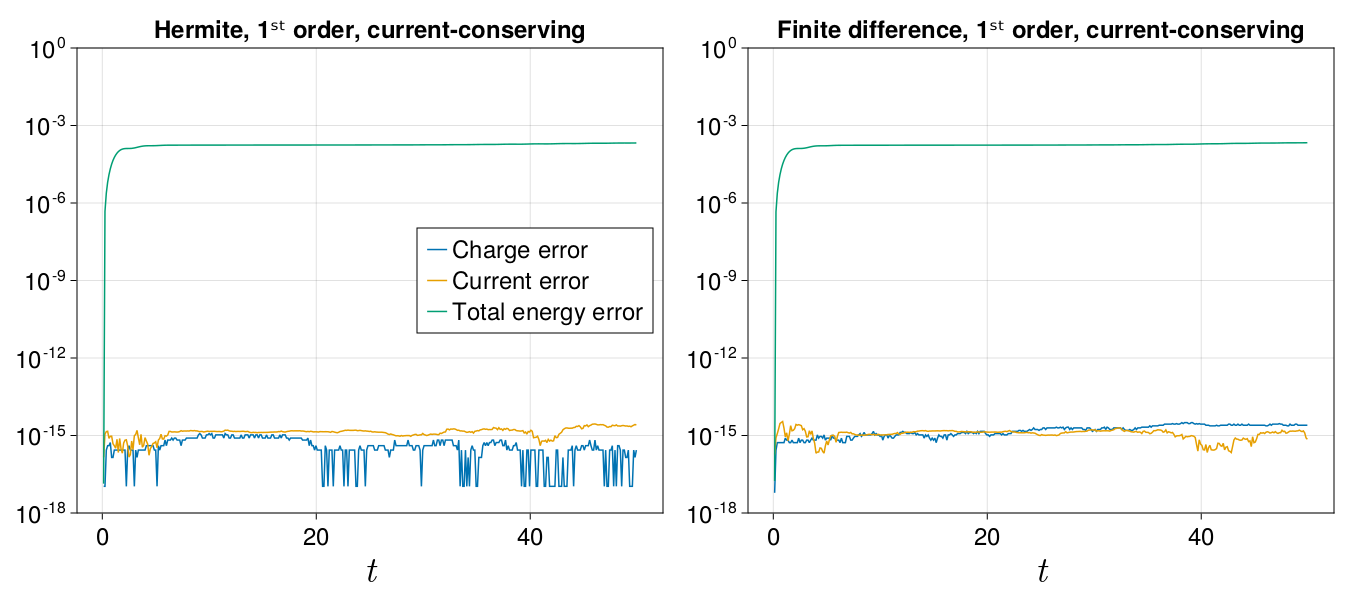}
        \caption{Exact conservation of charge and current by the first order integrator with a Poisson solve.}
    \end{subfigure}
    \begin{subfigure}{\textwidth}
        \centering
        \includegraphics[width=\textwidth]{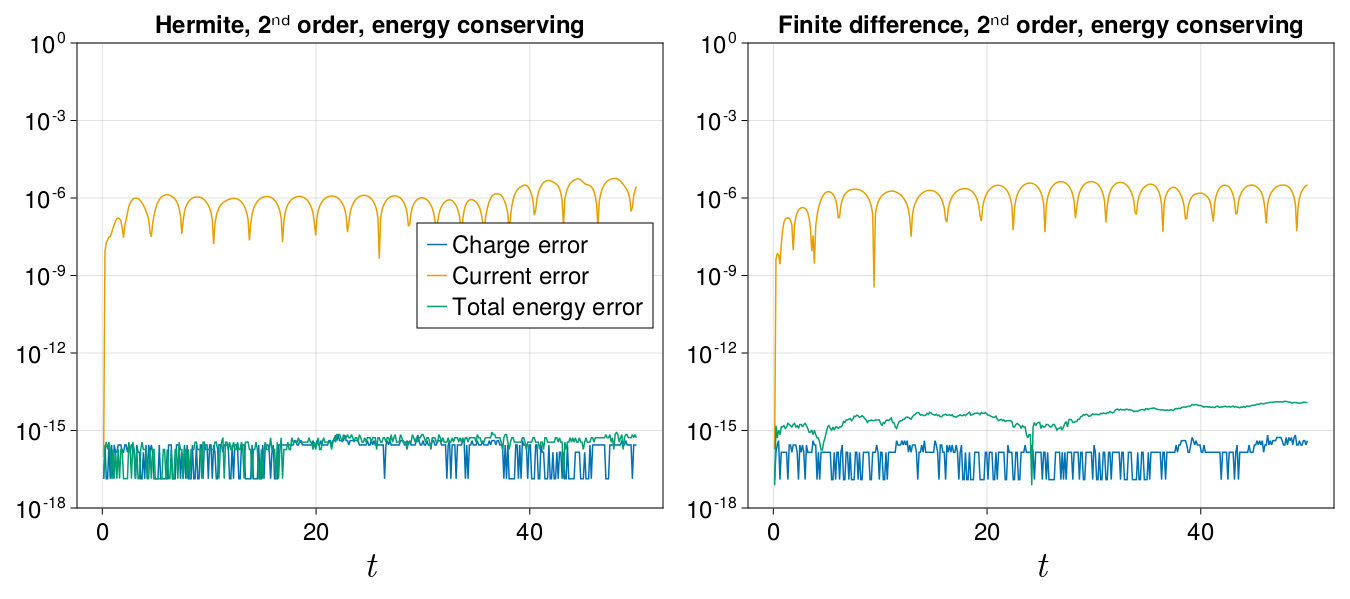}
        \caption{Exact conservation of charge and energy by the second order integrator with an \Ampere\ solve.}
    \end{subfigure}
    \caption{Conservation plots in the strong Landau damping example. We observe the benefit of overall second-order accuracy in the improved
    conservation of current in (b), compared to the conservation error of energy in (a). \label{fig:sld_conservation}}
\end{figure}

\subsection{Two-stream instability}
Here we reproduce the two-stream instability example from \cite{filbetConservativeDiscontinuousGalerkin2022}:
\begin{align*}
f(x, v, 0) = \frac{2}{7} (1 + 5v^2) (1 + \delta((\cos(2kx) + \cos(3kx)) / 1.2 + \cos(kx))) \frac{1}{\sqrt{2\pi}} e^{-v^2/2},
\end{align*}
with \( \delta = 0.01, k = 0.5 \).
This form of the distribution function is chosen to give the following analytic forms for the zeroth and second Hermite moments:
\begin{align*}
    f_0^{\text{Hermite}}(x, 0) = \frac{12}{7} (1 + \delta ((\cos(2kx) + \cos(3kx)) / 1.2 + \cos(kx))), \\
    f_2^{\text{Hermite}}(x, 0) = \frac{10\sqrt{2}}{7} (1 + \delta ((\cos(2kx) + \cos(3kx)) / 1.2 + \cos(kx))).
\end{align*}
We run this simulation on a grid with \( N_x = 256 \) grid points in \( x \).
For the Hermite spectral discretization we use \( M = 256 \) Hermite modes, and for the finite difference
discretization we use $N_v = 256$ velocity grid points.
The rank is set to \( r=20 \), and the instability is evolved with $\Delta t = \num{4e-3}$ well into the 
nonlinear phase, up to \( t = 50.0 \).
The results are shown in Figure \ref{fig:twostream}.
\begin{figure}
    \begin{subfigure}[h]{\textwidth}
        \centering
        \includegraphics[width=\textwidth]{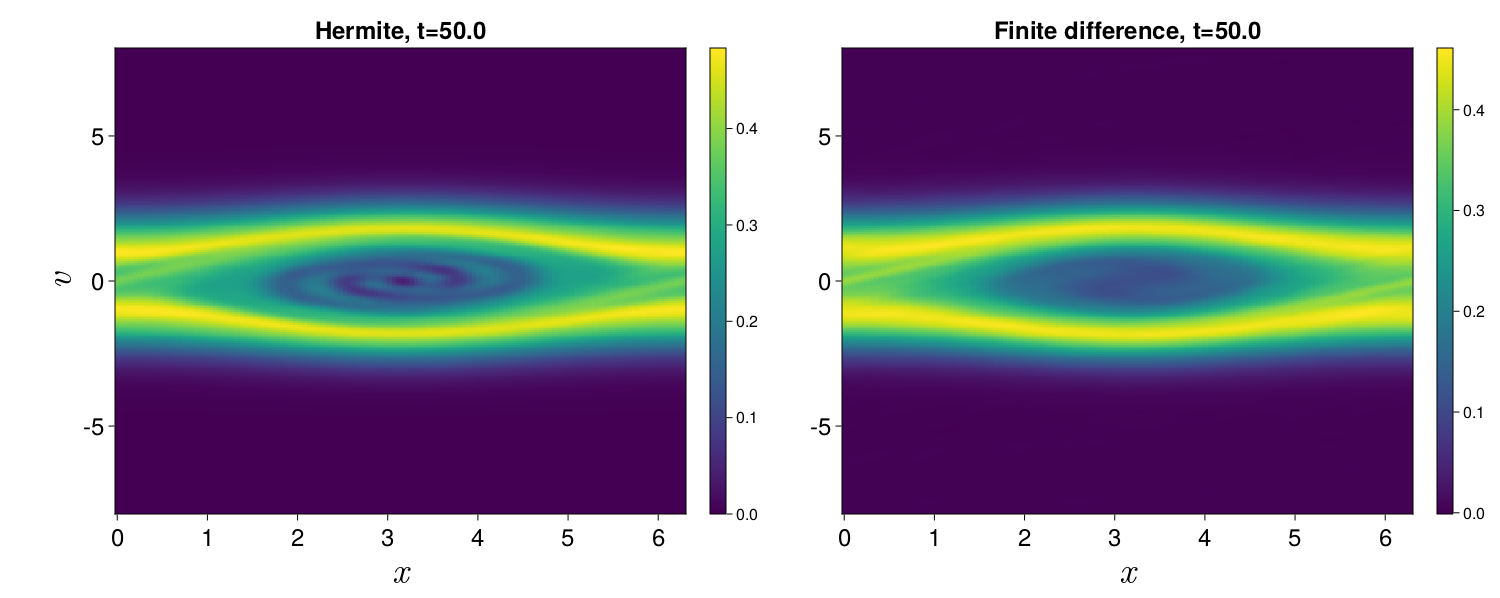}
        \caption{Plots of $f(x, v, t=50.0)$, depicting nonlinear phase space trapping in the two-stream instability.}
    \end{subfigure}
    \\
    \begin{subfigure}[h]{\textwidth}
        \centering
        \includegraphics[width=\textwidth]{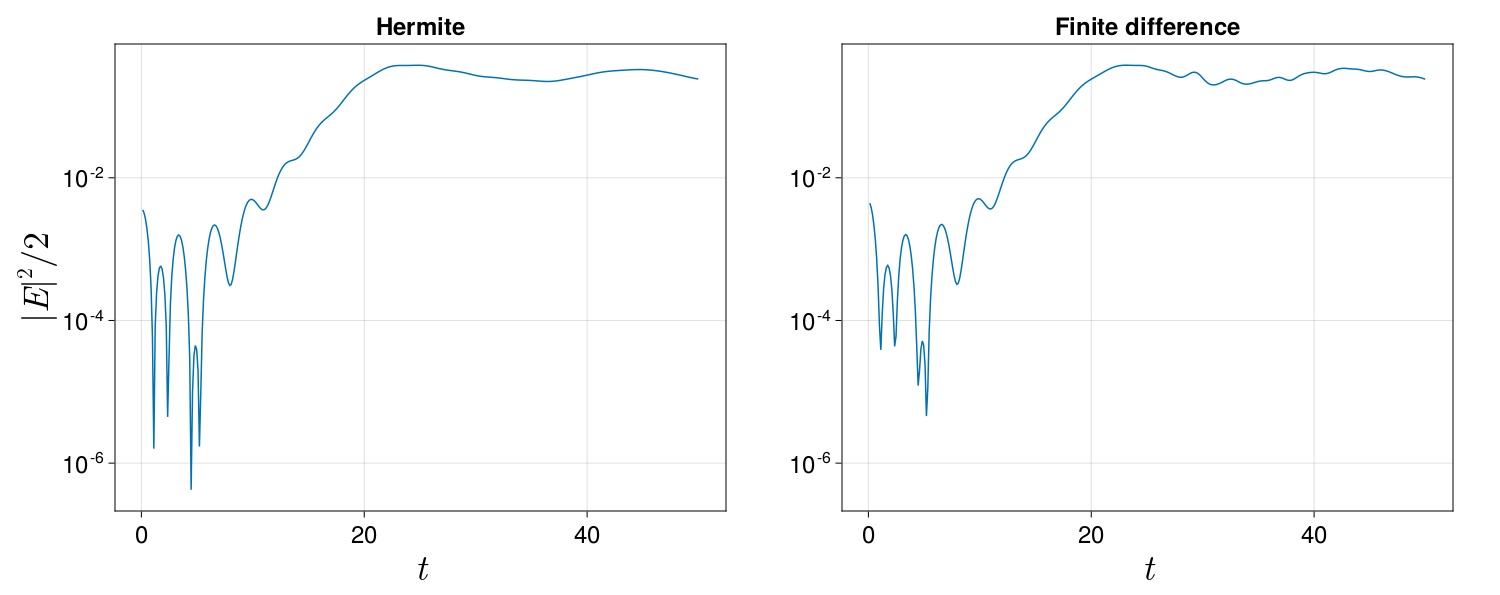}
        \caption{Electric energy trace. Our results show excellent agreement with published numerical results in \cite{filbetConservativeDiscontinuousGalerkin2022}.}
    \end{subfigure}
    \\
    \begin{subfigure}[h]{\textwidth}
        \centering
        \includegraphics[width=\textwidth]{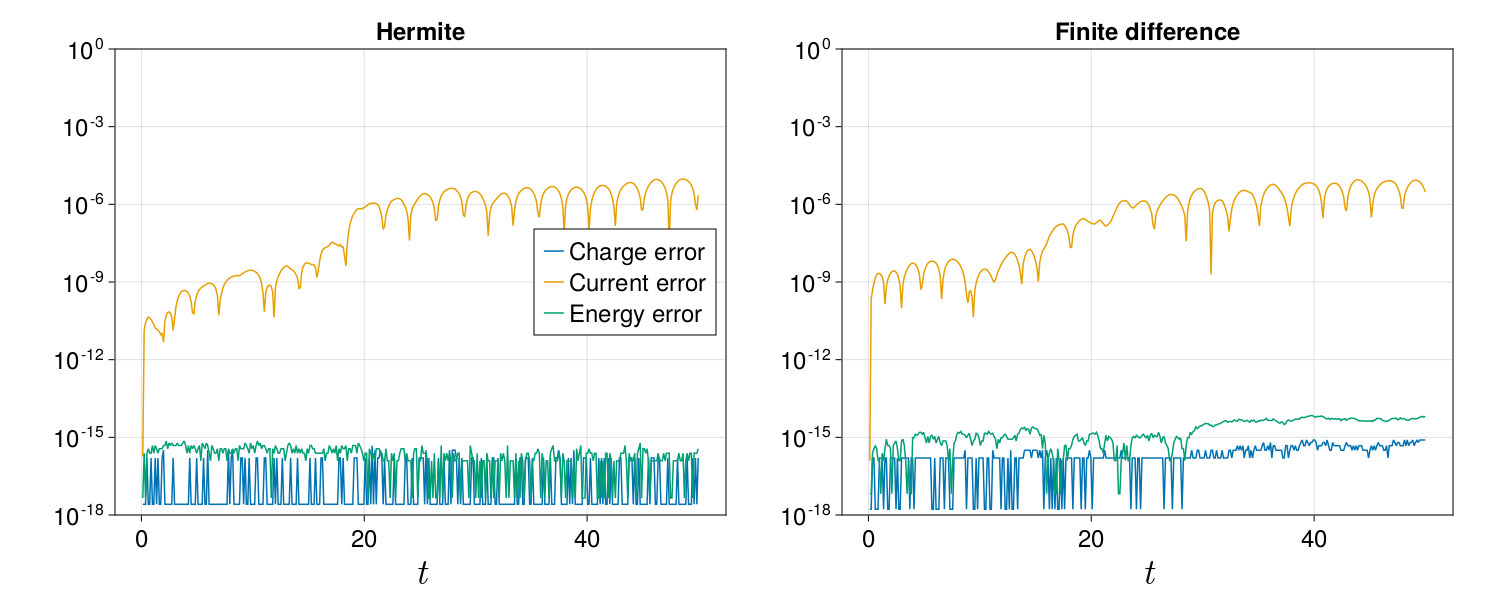}
        \caption{Exact conservation of charge and energy by the second order integrator.}
    \end{subfigure}
    \caption{Two-stream instability. \label{fig:twostream}}
\end{figure}

\section{Conclusion}
\label{sec:conclusion}
We have demonstrated a novel macro-micro decomposition which separates the particle distribution function $f$
into a rank-3 macroscopic portion which shares the moments of $f$, and a microscopic part which may be
evolved in the dynamical low-rank approximation framework.
This separation leads to a method which shares the efficiency benefits of the standard DLR approach while
preserving conservation of charge, current, and kinetic energy density.
Our macro-micro decomposition can be combined with appropriate temporal and spatial discretizations
to obtain schemes which exactly conserve charge and either current or total energy,
and exhibit second-order accuracy in time on our test problems.

To construct the decomposition, we use the orthogonal polynomial family corresponding to a weighted
inner product over velocity space to form an orthogonal projection which effectively separates the
macroscopic and microscopic portions of $f$.
Our approach has the benefit of supporting both infinite and truncated velocity domains.
Because the decomposition happens at the equation level, one can choose any discretization of velocity
space which is suitable for the application at hand.
To demonstrate this flexibility, we have implemented both a Hermite global spectral discretization
and a conservative finite difference discretization of velocity space.

As a proof of concept, we have implemented this scheme in one dimension, demonstrating the effectiveness
of the approach on standard plasma test problems.
We anticipate that extending the scheme to multiple dimensions should pose no essential difficulty,
since one can obtain a similar macro-micro decomposition based on tensor products of orthgonal polynomials.
Similarly, applying our scheme to the full Vlasov-Maxwell system would capture fully electromagnetic physics without
disproportionate complications.

\section*{Acknowledgements}
The authors would like to thank the anonymous reviewers for their feedback which greatly
improved the quality of this article.
We are especially indebted to the reviewer who discovered and helped us fix a hole in the
proofs of conservation.
The information, data, or work presented herein is based upon work supported by the 
National Science Foundation under Grant No. PHY-2108419. JH's research was also 
partially supported by AFOSR grant FA9550-21-1-0358 and DOE grant DE-SC0023164.

\bibliographystyle{plainnat}
\bibliography{references-bibtex.bib}

\end{document}